\documentclass[11pt]{article}

\newcommand\email[1]{\texttt{#1}}
\newcommand\affil{}
\usepackage{etex}

\usepackage[a4paper]{geometry}

\usepackage
[
colorlinks=true,
linkcolor=blue,
anchorcolor=blue,
citecolor=blue,
urlcolor=blue,
plainpages=false,
pdfpagelabels
]{hyperref}

\usepackage{microtype}

\usepackage[utf8]{inputenc}

\usepackage{graphicx}
\usepackage{amsmath}
\usepackage{mathrsfs}
\usepackage{amsthm}
\usepackage{thmtools}
\usepackage{mathtools}

\usepackage{enumerate}
\usepackage[capitalize]{cleveref}
\usepackage[noend]{algpseudocode}
\usepackage{comment}
\usepackage[marginpar]{todo}

\makeatletter
\renewcommand{\@tododisplay}[1]{%
\marginpar{#1}%
}
\renewcommand\@displaytodo[2][\todomark]{%
\@tododisplay{{\todoformat #1~(\ref{todolbl:\thetodo})}}%
\footnote[\thetodo]{#2}%
\global\@todotoks\expandafter{\the\@todotoks\todoitem{#1}{#2}}%
\@todotrue%
}%
\renewcommand\todomark{todo}
\makeatother

\usepackage{xcolor}
\usepackage[all]{xy}
\usepackage{mdwlist}

\usepackage[numbers,sort,compress]{natbib}
\usepackage[font={small,it}]{caption}

\usepackage{txfonts}

\usepackage{latexsym}
\usepackage{ifthen}
\usepackage{xspace}

\usepackage{tikz,tikz-cd}

\usepackage{epstopdf}

\AppendGraphicsExtensions{.svgz}

\AppendGraphicsExtensions{.ai}

\newtheorem{theorem}{Theorem}[section]
\newtheorem{proposition}[theorem]{Proposition}
\newtheorem{lemma}[theorem]{Lemma} 
\newtheorem{corollary}[theorem]{Corollary}

\theoremstyle{definition}
\newtheorem{example}[theorem]{Example}
\newtheorem{remark}[theorem]{Remark}

\newcommand\cat[1]{\ensuremath{\mathbf{#1}}}
\DeclareMathOperator{\im}{im}
\DeclareMathOperator{\coim}{coim}

\DeclareMathOperator{\Neg}{Neg}
\DeclareMathOperator{\coker}{coker}
\DeclareMathOperator{\obj}{obj}
\DeclareMathOperator{\rev}{rev}
\DeclareMathOperator{\rank}{rank}
\DeclareMathOperator{\dom}{coim}
\DeclareMathOperator{\id}{Id}
\DeclareMathOperator{\rad}{rad}

\newcommand{\kVect}[0]{\cat{Vect}}
\newcommand{\kvect}[0]{\cat{vect}}

\newcommand{\op}[0]{\mathrm{op}}
\newcommand{\R}[0]{\mathbb R}

\newcommand{\B}[1]{\mathcal B\ifthenelse{\equal{_{#1}}{}}{}{_{#1}}}
\newcommand{\Bopen}[1]{\mathcal U\ifthenelse{\equal{_{#1}}{}}{}{_{#1}}}
\newcommand{\undecorate}[0]{\xspace \mathcal U\xspace}

\newcommand{\barc}[1]{\mathcal X\ifthenelse{\equal{_{#1}}{}}{}{_{#1}}}

\newcommand{\dual}[0]{^*}
\newcommand{\doubledual}[0]{^{**}}

\newcommand{\D}[0]{\mathcal D}
\newcommand{\E}[0]{\mathbb E}

\newcommand{\I}[0]{\mathcal I}

\renewcommand{\S}[0]{\mathcal S}

\newcommand{\K}[0]{K}

\newcommand{\Nbb}[0]{\mathbb N}

\newcommand{\MatchingCat}[0]{\cat{Mch}}
\newcommand{\RCat}[0]{\cat{R}}

\newcommand{\Inj}[2]{\ensuremath{\alpha_{#1}^{#2}}}

\newcommand{\shift}[2]{{#1(#2)}}
\DeclareMathOperator{\ShiftOp}{Shift}
\newcommand{\Shift}[1]{{\ShiftOp_#1}}

\newcommand{\transition}[3]{\varphi_{#1}(#2,#3)}
\newcommand{\translation}[2]{\varphi_{#1}^{#2}}

\newcommand{\wildcard}{{\mkern2mu\cdot\mkern2mu}}

\newcommand{\birth}{b}
\newcommand{\death}{d}

\newcommand{\interval}[2]{{\langle #1 \mkern1mu, #2 \rangle}}
\newcommand{\intervalSub}[2]{{\langle #1 \mkern1mu , \mkern2.5mu #2 \rangle}}

\newcommand{\elements}[1]{{\mathrm{Rep}(#1)}}

\newcommand{\pfd}{p.f.d.\@\xspace}

\usetikzlibrary{matrix,arrows,positioning,calc,decorations.markings}
\tikzset{strike thru/.style={
    decoration={markings, mark=at position 0.5 with {
        \draw [-] 
            ++ ( 0pt,-1.25pt) 
            -- ( 0pt, 1.25pt);}
    },
    postaction={decorate},
}}

\makeatletter
\newcommand*{\matching}{%
  \mathrel{%
    \mathpalette\@vneq{\rightarrow}%
  }%
}
\newcommand*{\@vneq}[2]{%
  \sbox0{\raisebox{\depth}{$#1\neq$}}%
  \sbox2{\raisebox{\depth}{$#1|\,\m@th$}}%
  \ifdim\ht2>\ht0 %
    \sbox2{\resizebox{\vneqxscale\width}{\vneqyscale\ht0}{\unhbox2}}%
  \fi
  \sbox2{$\m@th#1\vcenter{\copy2}$}%
  \ooalign{%
    \hfil\phantom{\copy2}\hfil\cr
    \hfil$#1#2\m@th$\hfil\cr
    \hfil\copy2\hfil\cr
  }%
}
\newcommand*{\vneqxscale}{1}
\newcommand*{\vneqyscale}{.3}
\makeatother

\begin{document}
\title{Induced Matchings and the Algebraic Stability of~Persistence Barcodes}

\author{
Ulrich Bauer%
\thanks{\affil{Technische Universität München, Germany}.
\email{mail@ulrich-bauer.org}}
\ and Michael Lesnick%
\thanks{\affil{Institute for Mathematics and its Applications, %
Minneapolis, MN, USA}.
\email{mlesnick@ima.umn.edu}}
}

\date{January 26, 2015}

\clubpenalty=10000
\widowpenalty = 10000

\maketitle

\begin{abstract}
We define a simple, explicit map sending a morphism $f:M\to N$ of pointwise finite dimensional persistence modules to a matching between the barcodes of $M$ and $N$.  Our main result is that, in a precise sense, the quality of this matching is tightly controlled by the lengths of the longest intervals in the barcodes of $\ker{f}$ and $\coker{f}$. 

As an immediate corollary, we obtain a new proof of the \emph{algebraic stability theorem} for persistence barcodes \cite{chazal2009proximity,chazal2012structure}, a fundamental result in the theory of persistent homology.  In contrast to previous proofs, ours shows explicitly how a $\delta$-interleaving morphism between two persistence modules induces a $\delta$-matching between the barcodes of the two modules.  Our main result also specializes to a structure theorem for submodules and quotients of persistence modules, and yields a novel ``single-morphism" characterization of the interleaving relation on persistence modules.
\end{abstract}

\section{Introduction}\label{Sec:Intro}
Persistent homology, a topological tool for analyzing the global, non-linear, geometric features of data, is one of the primary objects of study in applied topology.  It provides simple, readily computed invariants, called \emph{barcodes}, of a variety of types of data, such as finite metric spaces and $\R$-valued functions.
A barcode is simply a collection of intervals in $\R$; we regard each interval as the lifespan of a topological feature of our data, and we interpret the length of the interval as a measure of the significance of that feature.  

To obtain a barcode from data, we proceed in three steps:
\begin{enumerate}
\item We first associate to the data a \emph{filtration}, a family of topological spaces $F=\{F_t\}_{t\in \R}$ such that $F_s\subseteq F_t$ whenever $s\leq t$.  For example, if our data is a function $\gamma:T\to \R$, we
may take $F=\S^\gamma$, where \[\S^\gamma_t=\{x\in T \mid \gamma(x)\leq t\};\]  we call $\S^\gamma$ the \emph{sublevel set filtration of $\gamma$}.  
\item We then apply $i^{\rm th}$ singular homology with coefficients in a field to each space $F_t$ and each inclusion $F_s\hookrightarrow F_t$.
This yields a \emph{persistence module} $H_i(F)$ for any $i\geq 0$, i.e., a diagram of vector spaces indexed by the totally ordered set $\R$.  
\item A persistence module $M$ is said to be \emph{pointwise finite dimensional (\pfd)} if each of the vector spaces $M_t$ is finite dimensional.  The structure theorem of \cite{crawley2012decomposition} yields a barcode invariant $\B M$ of any \pfd persistence $M$; $\B M$ specifies the decomposition of $M$ into indecomposable summands.  Thus, under mild assumptions on~$F$, we obtain a barcode invariant $\B {H_i(F)}$ of our data for each $i\geq 0$.  
\end{enumerate}
In the last fifteen years, these invariants have been applied widely to the study of scientific data \cite{edelsbrunner2010computational,Carlsson2014Topological}, and have been the subject of a great deal of theoretical interest.

\paragraph{The Algebraic Stability Theorem}
In 2005, Cohen-Steiner, Edelsbrunner, and Harer introduced a stability result for the persistent homology of $\R$-valued functions \cite{cohen2007stability}.  In brief, the result tells us that the map sending an $\R$-valued function to its persistence barcode is 1-Lipschitz with respect to suitable choices of metrics.    

In 2009, Chazal, Cohen-Steiner, Glisse, Guibas, and Oudot showed that the stability result of \cite{cohen2007stability} admits a purely algebraic generalization \cite{chazal2009proximity}.  This generalization, known as the \emph{algebraic stability theorem}, asserts that if there exists a \emph{$\delta$-interleaving} (a sort of ``approximate isomorphism'') between two \pfd persistence modules $M$, $N$, then there exists a \emph{$\delta$-matching} (approximate isomorphism) between the barcodes $\B{M}$, $\B{N}$ of these persistence modules.  

The algebraic stability theorem is perhaps the central  theorem in the theory of persistent homology.  It provides the core mathematical justification for the use of persistent homology in the study of noisy data.  The theorem is used, in one form or another, in nearly all available results on the approximation, inference, and estimation of persistent homology.  It has also been the basis for much subsequent theoretical work on persistence.   

While the earlier stability result for the persistent homology of $\R$-valued functions \cite{cohen2007stability} is itself a powerful result with several important applications, the more general algebraic stability theorem offers significant advantages.  First, it allows us to dispense with some technical conditions on the stability result for functions of \cite{cohen2007stability}, thereby yielding a generalization of that result to a wider class of functions.  Second and perhaps more importantly, the algebraic stability theorem offers a formalism for comparing the barcode invariants of functions defined on different domains, or for comparing barcode invariants of data that do not directly arise from functions at all.  In practice, this added flexibility is quite valuable: It allows us to establish fundamental approximation and inference theorems for persistent homology that would otherwise be much harder to come by \cite{chazal2009analysis,chazal2013clustering,chazal2009gromov,Chazal2013Persistence,edelsbrunner2013persistent}.  For one example, the algebraic stability theorem has been applied to establish the consistency properties of Rips complex-based estimators for the persistent homology of probability density functions \cite{chazal2013clustering}.

 \goodbreak
 
\paragraph{The Isometry Theorem}
In fact, the converse to the algebraic stability theorem also holds: There exists a {$\delta$-interleaving} between \pfd persistence modules $M$ and $N$ \emph {if and only if} there exists a {$\delta$-matching} between $\B{M}$ and $\B{N}$.  The algebraic stability theorem and its converse are together known as the \emph{isometry theorem}. 
A slightly weaker formulation of the isometry theorem establishes a relationship between the \emph{interleaving distance} (a pseudometric on persistence modules) and the \emph{bottleneck distance} (a widely studied pseudometric on barcodes): It says that the interleaving distance between $M$ and $N$ is equal to the bottleneck distance between $\B{M}$ and $\B{N}$.

Given the structure theorem for persistence modules \cite{crawley2012decomposition}, the converse algebraic stability theorem admits a very simple, direct proof.  The converse was first proven for \pfd persistence modules in \cite{lesnick2013theory}.  Later proofs appeared, independently, in \cite{chazal2012structure} (in a slightly more general setting), and in \cite{Bubenik2014Categorification} (in a special case). 
We give the proof of the converse algebraic stability theorem in \cref{Sec:ConverseToASP}, following \cite{lesnick2013theory}.  

The isometry theorem is interesting in part because the definition of the interleaving distance extends to a variety of generalized persistence settings where the direct definition of the bottleneck distance does not.  For example, interleaving distances can be defined on multidimensional persistence modules \cite{lesnick2013theory} and filtrations \cite{chazal2009proximity,lesnick2012multidimensional,blumberg2015universality}. 
The isometry theorem thus suggests one way to extend the definition of the bottleneck distance to these settings.  Since much of the theory of topological data analysis is formulated in terms of the bottleneck distance \cite{chazal2009analysis, chazal2013clustering, Chazal2013Persistence, balakrishnan2013statistical}, this opens the door to the adaptation of the theory to more general settings.  

Interleaving distances on multidimensional persistence modules and filtered topological spaces satisfy universal properties, indicating that they are, in certain relative senses, the ``right'' generalizations of the bottleneck distance  \cite{lesnick2013theory,lesnick2012multidimensional,blumberg2015universality}.

\paragraph{Earlier Proofs of the Algebraic Stability Theorem}
The original proof of the algebraic stability theorem is an algebraic adaptation of the stability argument for the persistent homology of functions given in \cite{cohen2007stability}.  Owing to its geometric origins, this argument has a distinctly geometric flavor.  In particular, the argument employs an (algebraic) interpolation lemma for persistence modules inspired by an interpolation construction for functions appearing in \cite{cohen2007stability}.  

In 2012, Chazal, Glisse, de Silva, and Oudot \cite{chazal2012structure} presented a reworking of the original proof of  algebraic stability as part of an 80-page treatise on persistence modules, barcodes, and the isometry theorem.  The reworked proof is similar on a high level to the original, but differs in the details.  In particular, it makes use of a characterization of barcodes in terms of \emph {rectangle measures}; these are functions from the set of rectangles in the plane to $\Nbb\cup \{\infty\}$ which have properties analogous to those of a measure.

The original proof of \cite{chazal2009proximity} and the proof given in \cite{chazal2012structure} are, to the best of our knowledge, the only proofs of the algebraic stability theorem in the literature.  
However, in unpublished work from 2011, Guillaume Troianowski and Daniel M\"ullner gave a third proof which, in contrast to the proofs of \cite{chazal2009proximity} and \cite{chazal2012structure}, establishes the theorem as a corollary of a general fact about persistence modules.  

First, in early 2011, Troianowski showed that if $f:M\to N$ is a morphism of persistence modules such that the length of each interval in the barcodes $\B{\ker{f}}$ and $\B{\coker{f}}$ is at most $\epsilon$, then the bottleneck distance between $\B{M}$ and $\B{N}$ is at most $2\epsilon$.  Troianowski never made this result public, but he mentioned it to the second author of this paper.  Troianowski's result implies that if the interleaving distance between $M$ and $N$ is $\delta$, then the bottleneck distance between $\B{M}$ and $\B{N}$ is at most $4\delta$.  The result thus implies the algebraic stability theorem, up to a factor of 4.

As we were putting the finishing touches on the present paper, we learned that in a September 2011 manuscript never made public \cite{mullner2011proximity}, M\"ullner and Troianowski strengthened Troianowksi's result to show that under the same assumptions on $\ker f$ and $\coker f$ as above, the bottleneck distance between $\B{\shift M {\frac{\epsilon}{2}}}$ and $\B{N}$ is at most $\frac{\epsilon}{2}$; here $\shift M {\frac{\epsilon}{2}}$ denotes the shift of the persistence module $M$ to the left by $\frac{\epsilon}{2}$.  This stronger result implies the algebraic stability theorem.  In turn, the results of the present paper, obtained independently of \cite{mullner2011proximity}, strengthen those of \cite{mullner2011proximity}.  

The proof of M\"{u}llner and Troianowski's result is an adaptation of the original proof of the algebraic stability theorem, and closely follows the technical details of that proof.  
The three existing proofs of algebraic stability are thus rather similar to one another.  In particular, each relies in an essential way on some version of the interpolation lemma.  

The interpolation lemma is interesting and pretty in its own right, and variants of it have found applications apart from the proof of algebraic stability \cite{lesnick2013theory,de2013geometry}.  However, reliance on the interpolation lemma makes the existing proofs of algebraic stability rather indirect: Given a pair of $\delta$-interleaved persistence modules $M$ and $N$, these proofs construct a $\delta$-matching between $\B{M}$ and $\B{N}$ only implicitly, and in a roundabout way, requiring several technical lemmas, a compactness argument, and consideration of a sequence of interpolating barcodes.
 
\subsection{Induced Matchings of Barcodes}
It is natural to ask whether there exists a more direct proof of the algebraic stability theorem which associates to a $\delta$-interleaving between two persistence modules a $\delta$-matching between their barcodes in a simple, explicit way.  

We present such a proof in this paper.  To do so, we define a map $\barc{}$ sending each morphism $f : M \to N$ of pointwise finite dimensional persistence modules to a matching $\barc f$ between the barcodes $\B{M}$ and $\B{N}$.  This map $\barc{}$ is not functorial in general; in fact, we prove that it is impossible to define such a map in a fully functorial way.  
However, $\barc{}$ is functorial on the two subcategories of persistence modules whose morphisms are the monomorphisms and the epimorphisms, respectively.
$\barc f$ is completely determined by the barcodes  $\B{M}$, $\B{N}$, and $\B{\im f}$, and also completely determines these three barcodes.

\paragraph{The Induced Matching Theorem}
We establish the algebraic stability theorem for pointwise finite dimensional persistence modules as an immediate corollary of a general result about the behavior of the matching~$\barc{f}$.  This result, which we call the \emph{induced matching theorem}, tells us that the quality of the matching $\barc{f}$ is tightly controlled by the lengths of the longest intervals in the barcodes $\B{\ker{f}}$ and $\B{\coker{f}}$. Roughly, it says that for any pair $(I,J)\in \B{M} \times \B{N}$  of intervals matched by $\barc{f}$,
\begin{enumerate}[(i)]
\item $J$ is obtained from $I$ by moving both the left and right endpoints of $I$ to the left,
\item if each interval in $\B{\ker f}$ is of length at most~$\epsilon$, then $\barc{f}$ matches all intervals in $\B{N}$ of length greater than $\epsilon$, and the right endpoints of $I$ and $J$ differ by at most $\epsilon$, and
\item dually, if each interval in $\B{\coker f}$ is of length at most $\epsilon$, then $\barc{f}$ matches all intervals in $\B{M}$ of length greater than $\epsilon$, and the left endpoints of $I$ and $J$ differ by at most $\epsilon$.
\end{enumerate}
We give the precise statement of the theorem in \cref{Sec:TheoremStatements}.
\begin{figure}[hbt]
\centerline{\includegraphics[scale=1.]{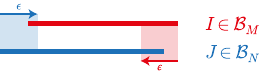}}
{\caption{
The relationship between a pair of intervals matched by $\protect\barc{f}$, as described by the induced matching theorem.
}
}
\end{figure}

\paragraph{A Structure Theorem for Persistence Submodules and Quotients}
Our proof of the induced matching theorem uses none of the technical intermediates appearing in the earlier proofs of the algebraic stability theorem.  %
Instead, our proof centers on an easy but apparently novel structure theorem for submodules and quotients of persistence modules.  This structure theorem is in fact the specialization of the induced matching theorem to the cases $\ker f=0$ and $\coker f=0$, corresponding to $\epsilon=0$ in (ii) and (iii) above.  A slightly weaker version of our structure theorem appears, independently, in the unpublished manuscript \cite{mullner2011proximity}, with a different proof.

Roughly, the structure theorem says that for \pfd persistence modules $K\subseteq M$, there exist canonical choices of 
\begin{enumerate}[(i)]
\item  an injection $\B{K}\hookrightarrow \B{M}$ mapping each interval $I\in\B{K}$ to an interval in $\B{M}$ which contains $I$ and has the same right endpoint as $I$,  
 \item an injection $\B{M/K}\hookrightarrow \B{M}$ mapping each interval $I\in\B{M/K}$ to an interval in $\B{M}$ which contains $I$ and has the same left endpoint as $I$.  
\end{enumerate}

\begin{figure}[hbt]
\centerline{\includegraphics[scale=1.]{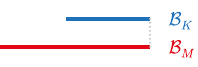} \hfil \includegraphics[scale=1.]{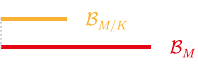}}
{\caption{Intervals in the barcodes $\protect\B{K}$ and $\protect\B{M/K}$, together with their images under the canonical injections $\protect\B{K}\hookrightarrow \protect\B{M}$ and $\protect\B{M/K}\hookrightarrow \protect\B{M}$.}
}
\end{figure}

Our definition of the matching $\barc{f}$ induced by a morphism $f:M\to N$ of persistence modules is derived from the definitions of these canonical injections in a simple way.  We consider the canonical factorization of $f$ through its image,
\[M\twoheadrightarrow \im f\hookrightarrow N.\]
Note that $\im{f}$ is a submodule of $N$, and up to isomorphism, $\im f$ is also a quotient of $M$.  Thus, the structure theorem gives us canonical injections $\B{M}\hookleftarrow \B{\im f}$ and $\B{\im f}\hookrightarrow \B{N}$, which we can interpret as matchings.  We define the matching $\barc{f}$ as a  composition of these two matchings.  See \cref{Sec:PartialMatchings,Sec:ProofOfStructureTheorem,Sec:InducedMatchingDef} for the formal definitions.

\begin{figure}[hbt]
\centerline{\includegraphics[scale=1.]{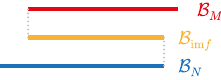}} 
{\caption{The matching $\protect\barc{f}$ between $\protect\B{M}$ and $\protect\B{N}$ is defined via the canonical injections $\protect\B{M}\hookleftarrow \protect\B{\im f}$ and $\protect\B{\im f}\hookrightarrow \protect\B{N}$.}
}
\end{figure}

\paragraph{A Single-Morphism Characterization of the Interleaving Relation}
As an easy corollary of the induced matching theorem and the converse algebraic stability theorem, we obtain a novel characterization of interleaved pairs of persistence modules: for any $\delta\geq 0$, two \pfd persistence modules $M$ and $N$ are $\delta$-interleaved if and only if there exists a morphism $f:M\to N(\delta)$ such that all intervals in $\B{\ker f}$ and $\B{\coker f}$ are of length at most $2\delta$.

\paragraph{The Isometry Theorem for q-Tame Persistence Modules}
Because of the structure theorem \cite{crawley2012decomposition} and the induced matching theorem, which hold for \pfd persistence modules, the setting of \pfd persistence modules is a very convenient one in which to formulate the isometry theorem.  In particular, in this setting we can formulate a sharp version of the isometry theorem, which appears here for the first time; see \cref{delta_matching_remark} and \cref{Thm:Isometry}.

On the other hand, as explained in \cite{chazal2012structure, Chazal2013Persistence}, there is good reason to want a version of the isometry theorem for \emph{q-tame persistence modules}.  These are persistence modules~$M$ for which the maps $M_s\to M_t$ are of finite rank whenever $s < t\in \R$.  Clearly, any \pfd persistence modules is q-tame.  q-tame persistence modules that are not necessarily \pfd arise naturally in the study of compact metric spaces and continuous functions on compact triangulable spaces \cite{chazal2012structure, Chazal2013Persistence}.  With this in mind, \cite[Theorem 4.11]{chazal2012structure} presents a version of the isometry theorem for q-tame persistence modules.  The chief difficulty in this, relative to the \pfd case, is in selecting a suitable definition the barcode of a q-tame persistence module; see \cref{Sec:QTameASP} and \cite{chazal2014observable}.

In \cref{Sec:QTameASP,Sec:ConverseToASP}, we show that \cite[Theorem 4.11]{chazal2012structure} follows easily from the isometry theorem for \pfd persistence modules.  In particular, we make clear in \cref{Sec:QTameASP} that the algebraic stability theorem for q-tame persistence modules is an easy corollary of the induced matching theorem.

\section{Preliminaries}\label{Sec:Preliminaries}
This section presents the basic definitions and notation
that we will use throughout the paper. 

\subsection{Persistence Modules and Barcodes}

\paragraph{Persistence Modules}\label{Sec:PersistenceModules}
For $\K$ a field, let $\kVect$ denote the category of vector spaces over~$\K$, and let $\kvect\subset \kVect$ denote the subcategory of finite dimensional vector spaces.
Let $\RCat$ denote the real numbers, considered as a poset category. That is, $\hom_\RCat(s,t)$ consists of a single morphism for $s\leq t$ and is empty otherwise.

As indicated in \cref{Sec:Intro}, a persistence module is a diagram of $\K$-vector spaces indexed by $\RCat$, %
i.e., a functor $\RCat\to\kVect$. 
The persistence modules form a category $\kVect^\RCat$ whose morphisms are natural transformations. 

Concretely, this means that a persistence module $M$ assigns to each $t\in \R$ a vector space~$M_t$, and to each pair $s\leq t\in \R$ a linear map $\transition M s t:M_s\to M_t$ in such a way that 
for all $r\leq s\leq t\in \R$, 
\[
\varphi_M(t,t)=\id_{M_t}
\qquad \text{and} \qquad
\varphi_M(s,t)\circ\varphi_M(r,s)=\varphi_M(r,t) .
\]
We call the maps $\varphi_M(s,t)$ \emph{transition maps}. In this notation, a morphism $f:M\to N$ of persistence modules is exactly a collection of maps $\{f_t:M_t\to N_t\}_{t\in \R}$ such that for all $s\leq t\in \R$ the following diagram commutes:
\[
\begin{tikzcd}[column sep=7.5ex]
  M_s \arrow{r}{\transition M s t} \arrow{d}[swap]{f_s}
& M_t \arrow{d}{f_t} \\
  N_s \arrow{r}[swap]{\transition N s t}
& N_t  
\end{tikzcd}
\]
We say $f$ is a monomorphism (epimorphism) if $f_t$ is an injection (surjection) for all $t\in \R$.

As already noted, we say $M$ is pointwise finite dimensional (\pfd) if $\dim M_t < \infty$ for all $t\in \R.$  Let $\kvect^\RCat$ denote the full subcategory of $\kVect^\RCat$ whose objects are \pfd persistence modules.  %

\paragraph{Interval Persistence Modules}\label{Sec:IntervalModules}
We say $I\subseteq \R$ is an interval if $I$ is non-empty and  $r,t\in I$ implies $s\in I$ whenever $r\leq s\leq t$.  Let ${\mathcal I_\R}$ denote the set of all intervals in $\R$.   
For $I\in \I_\R$, define the \emph{interval persistence module} $C(I)$ by
\begin{align*}
C(I)_t&=
\begin{cases}
\K &{\textup{if }} t\in I, \\
0 &{\textup{ otherwise}.}
\end{cases}\\
\transition {C(I)} s t&=
\begin{cases}
\id_\K &{\textup{if }} s,t\in I,\\
0 &{\textup{ otherwise}.}
\end{cases}
\end{align*}

\paragraph{Multisets}\label{Sec:Multisets}
Informally, a multiset is a collection where elements may repeat multiple times.  For example, $\{x,x,y\}$ is a multiset, where the element~$x$ appears with multiplicity 2 and the element~$y$ appears with multiplicity $1$.  

Formally, a multiset is a pair $(S,m)$, where $S$ is a set and $m:S\to \Nbb$ is a function.  Here $\Nbb$ denotes the positive integers.  One sometimes also allows $m$ to take values in a larger set of cardinals, but we will not do that here.  

For a multiset $\S=(S,m)$ define the \emph{representation} of $\S$ to be the set \[\elements\S=\{(s,k)\in S\times \Nbb \mid  k\leq m(s)\}.\]
We will generally work with multisets by way of their representations.

\paragraph{Barcodes}
We define a \emph{barcode} to be the representation of a multiset of intervals in $\R$.  In the literature, barcodes are also often called \emph{persistence diagrams}. Formally, the elements of a barcode are pairs $(I,k)$, with $I\in \I_\R$ and $k\in \Nbb$.    
In what follows, we will  denote an element $(I,k)$ of a barcode simply as $I$, suppressing the index $k$, and call such an element an interval.

Using \cite[Theorem 1]{azumaya1950corrections}, the authors of \cite{chazal2012structure} observe that in general, if $M$ is a (not necessarily \pfd) persistence module such that \[M\cong \bigoplus_{I\in \B{M}} C(I)\] for some barcode $\B{M}$, then $\B{M}$ is uniquely determined.  We say that $M$ is \emph{interval-decomposable} and call $\B{M}$ the \emph{barcode of~$M$}.
\begin{theorem}[Structure of \pfd Persistence Modules \cite{crawley2012decomposition}]
Any \pfd persistence module is interval-decomposable.  
\end{theorem}

\subsection{The Category of Matchings}\label{Sec:PartialMatchings}

A \emph{matching} from $S$ to $T$ (written as $\sigma:S\matching T$) is a bijection $\sigma:S' \to T'$, for some $S'\subseteq S$, $T'\subseteq T$; we denote $S'$ as $\dom \sigma$ and $T'$ as $\im \sigma$.
Formally, we regard $\sigma$ as a relation $\sigma\subseteq S\times T$, where $(s,t)\in \sigma$ if and only if $s\in \dom \sigma$ and $\sigma(s)=t$. %
  We define the \emph{reverse matching} $\rev{\sigma}:T\matching S$ %
  in the obvious way.    Note that any injective function is in particular a matching.  

For matchings $\sigma:S\matching T$ and $\tau:T\matching U$, we define the composition \(\tau\circ\sigma:S\matching U\) as
\begin{align*}
\tau\circ\sigma&=\{(s,u) \mid (s,t) \in \sigma, (t,u) \in \tau \text{ for some } t \in T \}.  
\end{align*}
As noted in \cite{edelsbrunner2013persistent}, with this definition we obtain a category $\MatchingCat$ whose objects are sets and whose morphisms are matchings.  

Given a matching $\sigma:S\matching T$, there is a canonical matching $S\matching \dom \sigma$; this matching is the categorical coimage of $\sigma$ in $\MatchingCat$.  Similarly, the canonical matching $\im \sigma \matching T$ is the categorical image of $\sigma$ in $\MatchingCat$.  This justifies our choice of notation for $\dom\sigma$ and $\im\sigma$.

\subsection{Decorated Endpoints and Intervals}
As noted in the introduction, our main results concern matchings of barcodes that move endpoints of intervals in controlled ways.  To make the notion of moving the endpoints of intervals precise in the $\RCat$-indexed setting, we need some formalism.  

Let $D=\{-,+\}$.  Adopting a notational and terminological convention of \cite{chazal2012structure}, we define the set~$\E$ of \emph{(decorated) endpoints} by $\E=\R \times D \cup \{-\infty,\infty\}$.  For $t\in \R$, we will write the decorated endpoints $(t,-)$ and $(t,+)$ as $t^-$ and $t^+$, respectively.   

We define a total order on $D$ by taking $- {} < {} +$.  The lexicographic total ordering on $\R \times D$ then extends to a total ordering on $\E$ by taking $-\infty$ to be the minimum element and taking $\infty$ to be the maximum element. 

We define an addition operation $(\wildcard) + (\wildcard)  :\E\times \R\to \E$ by taking
$s^\pm+t=(s+t)^\pm$ and $\pm\infty+t=\pm\infty$
for all $s,t\in \R$.  We define a subtraction operation $(\wildcard) - (\wildcard) :\E\times \R\to \E$ by taking $e-t=e+(-t)$ for $(e,t)\in \E\times \R$.    

There is a sensible bijection from the set $\{(\birth,\death)\in\E \times \E \mid \birth<\death\}$ to $\I_\R$, the set of intervals in $\R$, so that we may regard intervals as ordered pairs of decorated endpoints. This bijection is specified by the following table:
\[
\begin{tabular}{c|ccc}
&$t^-$&$t^+$&$\infty$\\
\hline
$-\infty$&$(-\infty,t)$&$(-\infty,t]$&$(-\infty,\infty)$\\
$s^-$&$[s,t)$&$[s,t]$&$[s,\infty)$\\
$s^+$&$(s,t)$&$(s,t]$&$(s,\infty)$\\
\end{tabular}
\]
For example, for $s<t\in \R$ the bijection sends $(s^+,t^-)$ to the interval $(s,t)$ and sends $( s^-,s^+)$ to the one-point interval $[s,s]$.
We will always denote the interval specified by $\birth<\death$ as $\interval{ \birth}{\death}$ to avoid confusion with the usual notation for open intervals.
Note that $\interval{ \birth}{\death} \subseteq \interval{ \birth'}{\death'}$ whenever $\birth'\leq \birth<\death\leq \death'$.

\section{The Isometry Theorem}\label{Sec:IsometryTheorem}
In this section we give the precise statement of the isometry theorem for \pfd persistence modules.   We discuss the isometry theorem for q-tame persistence modules in \cref{Sec:QTameASP,Sec:ConverseToASP}.
\subsection{Interleavings and the Interleaving Distance}
\label{InterleavingsOfModules} 

\paragraph{Distances}\label{Sec:NextSection}
An \emph{extended pseudometric} on a class %
$X$ is a function $d:X\times X\to [0,\infty]$ with the following three properties:
\begin{enumerate*}%
\item $d(x,x)=0$ for all $x\in X$,
\item $d(x,y)=d(y,x)$ for all $x,y\in X$,
\item $d(x,z)\leq d(x,y)+d(y,z)$ for all $x,y,z\in X$ such that $d(x,y),d(y,z) < \infty$.
\end{enumerate*}
Note that an extended pseudometric $d$ is a metric if $d$ is finite and $d(x,y)=0$ implies $x=y$.
In this paper, by a \emph{distance} we will mean an extended pseudometric.  

\paragraph{Shift Functors}
For $\delta\in \R$, we define the \emph{shift functor} $\shift{(\wildcard)}{\delta}:\kVect^\RCat \to \kVect^\RCat$ as follows: For $M$ a persistence module we let $\shift{M}{\delta}$ be the persistence module such that for all $t\in \R$ we have $\shift{M}{\delta}_t=M_{t+\delta}$, and for all $s\leq t\in \R$ we have $\transition {\shift{M}{\delta}} s t = \transition M {s+\delta} {t+\delta}$.
For a morphism $f\in \hom(\kVect^\RCat)$ we define $f(\delta)$ by taking $\shift{f}{\delta}_t=f_{t+\delta}$.
Note that the barcode $\B {M(\delta)}$ is obtained from $\B M$ by shifting all intervals to the left by $\delta$, as in \cref{Fig:Shift}.
\begin{figure}[hbt]
\centerline{\includegraphics[scale=1.]{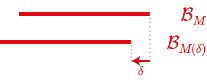}}
{\caption{
Corresponding intervals in $\protect\B{M}$ and $\protect\B{M(\delta)}$.
}
\label{Fig:Shift}
}
\end{figure}

\paragraph{Transition Morphisms}
For $M$ a  persistence module and $\delta\geq 0$, let the \emph{$\delta$-transition morphism} $\translation M \delta :M\to \shift{M}{\delta}$ be the morphism whose restriction to $M_t$ is the linear map $\transition M t {t+\delta}$, for all $t\in \R$.  

\paragraph{Interleavings}
We say that two persistence modules $M$ and $N$ are \emph{$\delta$-interleaved} if there exist morphisms $f:M\to \shift{N}{\delta}$ and $g:N\to \shift{M}{\delta}$ such that 
\begin{align*}
\shift g \delta \circ f&=\translation M {2\delta},\\
\shift f \delta \circ g&=\translation N {2\delta}. 
\end{align*}
We refer to such $f$ and $g$ as \emph{$\delta$-interleaving} morphisms. 
The definition of $\delta$-interleaving morphisms was introduced in \cite{chazal2009proximity}.

\begin{remark}\label{EasyInterleavingRemark}
It is easy to show that if $L$ and $M$ are $\delta$-interleaved, and $M$ and $N$ are $\delta'$-interleaved, then $L$ and $N$ are $(\delta+\delta')$-interleaved.  Similarly, if $0\leq\delta \leq \delta'$ and $M$ and $N$ are $\delta$-interleaved, then $M$ and $N$ are also $\delta'$-interleaved.  
\end{remark}
 
\begin{example}\label{Ex:InterleavingsFromFunctions}  Here is one of the central examples of a $\delta$-interleaving in topological data analysis: For $T$ a topological space and functions $\gamma, \kappa:T\to \R$, let \[d_\infty(\gamma,\kappa)=\sup_{y\in T} |\gamma(y)-\kappa(y)|.\]  Suppose $d_\infty(\gamma,\kappa)=\delta$.
Then for each $t\in \R$, we have inclusions
\begin{align*}
\S^{\gamma}_t\subseteq \S^{\kappa}_{t+\delta},\\
\S^{\kappa}_t\subseteq \S^{\gamma}_{t+\delta}.
\end{align*}
Applying the $i^\text{th}$ homology functor with coefficients in $\K$ to the collection of all such inclusion maps yields a $\delta$-interleaving between $H_i(\S^{\gamma})$ and $H_i(\S^{\kappa})$.
\end{example}

\paragraph{The Interleaving Distance}

We define $d_I:\obj(\kVect^\RCat)\times \obj(\kVect^\RCat)\to [0,\infty]$, the \emph{interleaving distance}, by taking 
\[d_I(M,N)=\inf\, \{\delta\in [0,\infty) \mid M\textup{ and  }N\textup{ are }\delta\textup{-interleaved}\}.\]
It is not hard to check that $d_I$ is a distance on persistence modules.  In addition, if $M$, $M'$, and $N$ are persistence modules with $M\cong M'$, then $d_I(M,N)=d_I(M',N)$, so $d_I$ descends to a distance on isomorphism classes of persistence modules.

\subsection{$\large\boldsymbol\delta$-Matchings and the Bottleneck Distance}\label{Sec:BottleneckDistance}
For $\D$ a barcode and $\epsilon\geq 0$, let 
\[\D^{\epsilon} = \{ \interval{\birth}{\death } \in \D \mid \birth + \epsilon < \death \}=\{I\in \D\mid [t,t+\epsilon]\subseteq I \textup{ for some }t\in \R\}.\]
Note that $\D^{0}=\D$. 
We define a \emph{$\delta$-matching} between barcodes $\mathcal C$ and $\mathcal D$ to be a matching $\sigma:\mathcal C\matching \mathcal D$ such that
\begin{enumerate*}%
\item $\mathcal C^{2\delta} \subseteq \dom{\sigma}$,
\item $\mathcal D^{2\delta}\subseteq \im \sigma$,
\item if 
$\sigma\interval{ \birth}{\death }=\interval{ \birth'}{\death' }$, then
\begin{align*}
\interval{ \birth}{\death }&\subseteq \interval{ \birth'-\delta}{\death'+\delta },\\
\interval{ \birth'}{\death' }&\subseteq\interval{ \birth-\delta}{\death+\delta }.
\end{align*}
\end{enumerate*}
See \cref{Fig:Bottleneck} for an example.  

\begin{remark}\label{Rmk:MatchingsTriangleIneq}
Note that for barcodes $\mathcal C$, $\mathcal D$, and $\mathcal E$, $\sigma_1:\mathcal C\matching \mathcal D$ a $\delta_1$-matching, and $\sigma_2:\mathcal D\matching \mathcal E$ a $\delta_2$-matching, $\sigma_2\circ \sigma_1:\mathcal C\to \mathcal E$ is a $(\delta_1+\delta_2)$-matching.  
\end{remark}

We define the bottleneck distance $d_B$ by
\[d_B(\mathcal C,\mathcal D)=\inf\, \{\delta\in [0,\infty) \mid \exists\textup{ a }\delta\textup{-matching between }\mathcal C\textup{ and }\mathcal D\}.\]
The triangle inequality for $d_B$ follows immediately from \cref{Rmk:MatchingsTriangleIneq}.

$d_B$ is the most commonly considered distance on barcodes in the persistent homology literature.  This is in part because $d_B$ is especially well behaved from a theoretical standpoint.

\begin{figure}[hbt]
\centerline{\includegraphics[scale=1.]{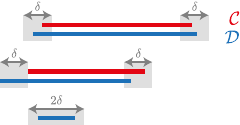}}
{\caption{
A $\delta$-matching between two barcodes $\mathcal C$ and $\mathcal D$.  Endpoints of matched intervals are at most $\delta$ apart, and unmatched intervals are of length at most $2\delta.$}
\label{Fig:Bottleneck}
}
\end{figure}

\begin{remark}\label{delta_matching_remark}
Our definition of a $\delta$-matching is slightly stronger than the one appearing in \cite[Section 4.2]{chazal2012structure}, which is insensitive to the decoration of the endpoints of intervals.  Using the stronger definition of $\delta$-matching allows us to state slightly sharper stability results.  However, regardless of which definition of $\delta$-matching one uses, the definition of the bottleneck distance one obtains is the same.   
\end{remark}
\paragraph{Bottleneck Distance as an Interleaving Distance}
As an aside, note that we can regard a barcode $\mathcal D$ as an object in $\MatchingCat^\RCat$, the category whose objects are functors $\RCat \to \MatchingCat$ and whose morphisms are natural transformations: For each real number~$t$ we let $\mathcal D_t$ be the subset of $\mathcal D$ consisting of all intervals which contain $t$, and for each~$s\leq t$ we define the transition matching $\phi_{\mathcal D}(s,t): \mathcal D_s \matching \mathcal D_t$ 
to be the identity on $\mathcal D_s\cap \mathcal D_t$.

Viewing barcodes as objects of $\MatchingCat^\RCat$ in this way, it is possible to define interleavings and an interleaving distance on barcodes in essentially the same way as for persistence modules.  The reader may check that $\delta$-interleavings of barcodes as thus defined correspond exactly to $\delta$-matchings of barcodes.  Thus, $d_B$ is equal to the interleaving distance on barcodes.
See \cite{Bauer2015Persistence} for details.

\subsection{The Isometry Theorem}

\begin{theorem}[Isometry Theorem for \pfd Persistence Modules]\label{Thm:Isometry}
Two \pfd persistence modules $M$ and $N$ are $\delta$-interleaved if and only if there exists a $\delta$-matching between $\B M$ and $\B N$. In particular, \[d_B(\B M,\B N)=d_I(M,N).\]
\end{theorem}

As noted earlier, the \emph{algebraic stability theorem} (\cref{Thm:pfdASP})  is the ``only if'' half of the isometry theorem: It says that if $M$ and $N$ are $\delta$-interleaved, then there exists a $\delta$-matching between $\B{M}$ and $\B{N}$.  We give a proof of the algebraic stability theorem as an immediate consequence of our induced matching theorem in \cref{Sec:ASP}. In \cref{Sec:ConverseToASP}, we present a proof of the converse, following~\cite{lesnick2013theory}.

As one illustration of the utility of the algebraic stability theorem, note that in view of \cref{Ex:InterleavingsFromFunctions}, the theorem yields the following strengthening of the original stability result for barcodes of $\R$-valued functions \cite{cohen2007stability}:

\begin{corollary}[%
\cite{cohen2007stability,chazal2009proximity}]\label{Cor:Original_Stability}
For any topological space $T$, functions $\gamma, \kappa:T\to \R$, and $i\geq 0$ such that $H_i(\S^\gamma)$ and $H_i(\S^\kappa)$ are \pfd, 
\[d_B(\B{H_i(\S^\gamma)},\B{H_i(\S^\kappa)})\leq d_\infty(\gamma,\kappa).\]
\end{corollary}

\subsection{Dualization of Persistence Modules}
\label{Sec:dualization}
To close this section on preliminaries, we examine the behavior of barcodes under dualization of persistence modules.  %
For more on duality in the context of persistence, see \cite{cohensteiner2008extending,carlsson2009zigzag,de2011dualities}.

Let $\RCat^{\op}$ denote the opposite category of $\RCat$ and let $\Neg: \RCat \to \RCat^{\op}$ be the isomorphism of categories such that for all objects $t\in \R$, $\Neg(t)=-t.$
Given any persistence module $M$, taking the dual of all vector spaces and all linear maps in $M$ gives us a functor $M^\dag:\RCat^{\op}\to\kVect$.  We define $M\dual$, the \emph{dual persistence module of $M$}, by $M\dual=M^\dag\circ \Neg:\RCat\to \kVect$.  If $M$ is \pfd, $M\doubledual$ is canonically isomorphic to $M$.

If $f:M\to N$ is a morphism of persistence modules, we define $f\dual:N\dual\to M\dual$ by taking $(f\dual)_t=(f_{-t})\dual$.
With these definitions, $(\wildcard)\dual$ is a contravariant endofunctor on $\kVect^\RCat$.
When $M$ and $N$ are \pfd, under the canonical identifications of $M\doubledual$ with $M$ and $N\doubledual$ with $N$, we have $f\doubledual=f$.

For $\D$ a barcode, let $\D\dual$ be the barcode $\{-I\mid I\in \D\}$, where we define $-I=\{t\mid -t\in I\}$ for any interval $I$ .
We then have the following easy observation, which we state without proof:
\begin{proposition}\label{Prop:BarcodesUnderDuality}
If $M$ is \pfd, then $\B{M\dual}=(\B{M})\dual$.
\end{proposition}

\section{The Structure of Persistence Submodules and Quotients}
\label{Sec:ProofOfStructureTheorem}
As a starting point for the main results of this paper, in this section we describe the relationship between the barcode of a \pfd persistence module~$M$ and the barcode of a submodule or quotient of $M$.    

\subsection{Canonical Injections}
To prepare for the main result of the section, we first introduce the definition of a \emph{canonical injection} between barcodes.

\paragraph{Canonical Injections between Enumerated Sets}

Define an \emph{enumerated set} $S$ to be a totally ordered set $S$ such that there exists an order-preserving bijection $S\to \Nbb$ or $S\to \{1,\dots,n\}$.  
For $S,T$ enumerated sets with $|S|\leq |T|$, we define the \emph{canonical injection} $\Inj{T}{S}:S\hookrightarrow T$ by 
\[
\Inj{T}{S}(s_i)=t_i,
\]
where $s_i$ and $t_i$ denote the $i^\text{th}$ element of $S$ and $T$, respectively.
\begin{remark}\label{Lem:Functoriality}
Clearly, for $S,T,U$ enumerated sets with $|S|\leq |T|\leq |U|$, the canonical injections satisfy \[\Inj{U}{S}=\Inj{U}{T}\circ\Inj{T}{S}\]
\end{remark}

\paragraph{Partitions of Barcodes into Enumerated Sets}
Suppose $\S=(S,m)$ is a multiset.  A total order on $S$ induces a canonical total order on $\elements\S$, which is obtained by restricting the lexicographic total order on $S\times \Nbb$ to~$\elements\S$.

For $M$ a \pfd persistence module and $\birth\in \E$, let ${\intervalSub{ \birth}{\wildcard}}_M$ denote the intervals in $\B M$ of the form $\interval{ \birth}{\death} $ for some $\death\in \E$.  Symmetrically, for $\death\in \E$, let ${\intervalSub{\wildcard}{\death}}_M$ denote the intervals in~$\B{M}$ of the form $\interval{ \birth}{\death} $ for some $\birth\in \E$.  Note that \[\B{M}=\coprod_{\death\in \E} {\intervalSub{\wildcard}{\death}}_M=\coprod_{\birth\in \E} {\intervalSub{\birth}{\wildcard}}_M.\]

For each $\birth,\death\in \E$, we regard both ${\intervalSub{ \birth}{\wildcard}}_M$  and ${\intervalSub{\wildcard}{\death}}_M$ as totally ordered sets, with the total order on each set induced by the reverse inclusion relation on intervals, so that larger intervals are ordered before smaller ones.  Thus, for example, if $\interval{ \birth}{\death}, \interval{ \birth}{\death'} \in \B{M}$ and $\death'>\death$, we have that $\interval{ \birth}{\death'}<\interval{ \birth}{\death}$ in the total order on ${\intervalSub{ \birth}{\wildcard}}_M$.
With these choices of total orders, each of the sets ${\intervalSub{ \birth}{\wildcard}}_M$ and ${\intervalSub{\wildcard}{\death}}_M$ is an enumerated set.

\paragraph{Canonical Injections Between Barcodes}
If $M$, $N$ are \pfd persistence modules and $|{\intervalSub{\wildcard}{\death}}_M|\leq |{\intervalSub{\wildcard}{\death}}_N|$ for each $\death\in \E$, then the canonical injections ${\intervalSub{\wildcard}{\death}}_M\hookrightarrow{\intervalSub{\wildcard}{\death}}_N$ assemble into an injection $\B M\hookrightarrow \B N$, which we also call a canonical injection.  Note that this injection maps the $i^{th}$ largest interval of ${\intervalSub{\wildcard}{\death}}_M$ to the $i^{th}$ largest interval of ${\intervalSub{\wildcard}{\death}}_N$ for all $\death \in \E$, $1\leq i\leq |{\intervalSub{\wildcard}{\death}}_M|$, as in \cref{Fig:Canononical_Inj}. 

\begin{figure}[hbt]
\centerline{\includegraphics[scale=1.]{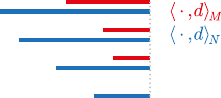}}
{\caption{
The canonical injection ${\intervalSub{\wildcard}{\death}}_M\hookrightarrow{\intervalSub{\wildcard}{\death}}_N$ matches intervals in order of decreasing length.   
}
\label{Fig:Canononical_Inj}
}
\end{figure}

Dually, if we have $|{\intervalSub{\birth}{\wildcard}}_M|\geq |{\intervalSub{\birth}{\wildcard}}_N|$ for each $b\in \E$, then the canonical injections ${\intervalSub{\birth}{\wildcard}}_M\hookleftarrow{\intervalSub{\birth}{\wildcard}}_N$ assemble into a canonical injection $\B M\hookleftarrow \B N$ which maps the $i^{th}$ largest interval of ${\intervalSub{\birth}{\wildcard}}_N$ to the $i^{th}$ largest interval of ${\intervalSub{\birth}{\wildcard}}_M$.

\subsection{Structure Theorem}
We now turn to the main result of this section, our structure theorem for submodules and quotients of persistence modules.

\begin{theorem}[Structure of Persistence Submodules and Quotients]
\label{Prop:StructPropForQuotientsAndSubmodules}
Let $M$ and $N$ be \pfd persistence modules.  
\begin{enumerate}[(i)]
\item If there exists a monomorphism $M\hookrightarrow N$, then 
for each $\death\in \E$ \[|{\intervalSub{\wildcard}{\death}}_M|\leq |{\intervalSub{\wildcard}{\death}}_N|,\]
and the canonical injection $\B{M}\hookrightarrow \B{N}$ %
maps each  $\langle \birth, \death\rangle \in \B{M}$ to an interval $\langle \birth',\death\rangle$ with $\birth' \leq \birth$.
\item Dually, if there exists an epimorphism $M\twoheadrightarrow N$, then 
for each $\birth\in \E$
\[|{\intervalSub{\birth}{\wildcard}}_M|\geq |{\intervalSub{\birth}{\wildcard}}_N|,\]
and the canonical injection $\B{M}\hookleftarrow \B{N}$ %
maps each interval $\langle \birth, \death\rangle\in \B{N}$ to an interval $\langle \birth, \death'\rangle$ with $\death\leq \death'$. 
\end{enumerate}
\end{theorem}

For $M$ a \pfd persistence module and $I=\interval{\birth}{\death}$,
let ${\intervalSub{\wildcard}{I}}_M\subseteq \B{M}$ denote the intervals in ${\intervalSub{\wildcard}{\death}}_M$ which contain $I$.  The key step in our proof of \cref{Prop:StructPropForQuotientsAndSubmodules}\,(i) is the following.

\begin{lemma}\label{Lem:StructPropLemma}
If $j:M\to N$ is a monomorphism between \pfd persistence modules, then for each interval $I$,
\[|{\intervalSub{\wildcard}{I}}_M|\leq |{\intervalSub{\wildcard}{I}}_N|.\]
\end{lemma}

\begin{proof}
Write $I=\interval{\birth}{\death}$.  
We may assume without loss of generality that \[M=\bigoplus_{J\in \B M} C(J),\qquad N=\bigoplus_{J\in \B N} C(J).\]  
Let \[U=\bigoplus_{J\in {\intervalSub{\wildcard}{I}}_M} C(J), \qquad V=\bigoplus_{J\in {\intervalSub{\wildcard}{I}}_N} C(J).\]
Clearly, $U\subseteq M$ and $V\subseteq N$.

For $J$ any interval and $t\in\R$, we write $t>J$ if $t>s$ for all $s\in J$.  
Since $M$ and $N$ are \pfd, there exists some $t\in I$ such that $t > J$ for all intervals $J=\interval{ \birth'}{\death' }\in \B{M}\cup \B{N}$ with $\birth' \leq \birth$ and $\death' < \death$.
Note that $\dim U_t=|\intervalSub{\wildcard}{I}_M|$ and 
$\dim V_t=|\intervalSub{\wildcard}{I}_N|$.

We claim that $j_t(U_t)\subseteq V_t$.
To see this, 
note that by the choice of~$t$,
\begin{align*}
U_t&=\bigcap_{\substack{s\in I\\ s\leq t}} \im \transition M s t\ \cap\  \bigcap_{r > I} \ker \varphi_M(t,r),\\
V_t&=\bigcap_{\substack{s\in I\\ s\leq t}} \im \transition N s t\ \cap\ \bigcap_{r > I} \ker \varphi_N(t,r).
\end{align*}
For each $s\in I$ such that $s\leq t$,
we have $j_t( \im \transition M s t)\subseteq  \im \transition N s t$ by the commutativity of the following diagram:
\[
\begin{tikzcd}[column sep=7.5ex]
  M_s \arrow{r}{\transition M s t} \arrow{d}[swap]{j_s}
& M_t \arrow{d}{j_t} \\
  N_s \arrow{r}[swap]{\transition N s t}
& N_t  
\end{tikzcd}
\]
Similarly, for each $r > I$, we have $j_t(\ker \varphi_M(t,r))\subseteq \ker \varphi_N(t,r)$.  Thus $j_t(U_t)\subseteq V_t$ as claimed.  
Since $j_t$ is an injection, we have $\dim U_t\leq \dim V_t$,
and the lemma follows.
\end{proof}

\begin{proof}[Proof of \cref{Prop:StructPropForQuotientsAndSubmodules}]
To show that $|{\intervalSub{\wildcard}{\death}}_M|\leq |{\intervalSub{\wildcard}{\death}}_N|,$
 it is enough to observe that for all $i\leq |{\intervalSub{\wildcard}{\death}}_M|$, we have $i\leq |{\intervalSub{\wildcard}{\death}}_N|$.  
 Let $I=\langle \birth, \death\rangle$ be the $i^\text{th}$ interval of the enumerated set ${\intervalSub{\wildcard}{\death}}_M$.  Note that for $1\leq j\leq i$, ${\intervalSub{\wildcard}{I}}_M$ contains the $j^\text{th}$ interval of ${\intervalSub{\wildcard}{\death}}_M$.  Hence
\[i \leq |{\intervalSub{\wildcard}{I}}_M|\leq  |{\intervalSub{\wildcard}{I}}_N|\leq |{\intervalSub{\wildcard}{d}}_N|,\]
where the second inequality follows from \cref{Lem:StructPropLemma}.  

Now let $I'$ denote the $i^\text{th}$ interval of ${\intervalSub{\wildcard}{\death}}_N$.  Recall that the canonical injection $\B{M}\hookrightarrow \B{N}$ matches $I$ to $I'$.  Since $|{\intervalSub{\wildcard}{I}}_M| \leq |{\intervalSub{\wildcard}{I}}_N|$ we must have $I'\in {\intervalSub{\wildcard}{I}}_N$, so $I'=\langle \birth', \death\rangle$ for some $\birth' \leq \birth$.
 This establishes \cref{Prop:StructPropForQuotientsAndSubmodules}\,(i).

\cref{Prop:StructPropForQuotientsAndSubmodules}\,(ii) follows from \cref{Prop:StructPropForQuotientsAndSubmodules}\,(i) by a simple duality argument:
Suppose $q:M\to N$ is an epimorphism of \pfd persistence modules.  Then $q\dual:N\dual\to M\dual$ is a monomorphism. %
\Cref{Prop:StructPropForQuotientsAndSubmodules}\,(i) yields the canonical injection $\B{N\dual}\hookrightarrow \B{M\dual}$ mapping
each $\langle \birth, \death\rangle \in \B{N\dual}$ to an interval $\langle \birth',\death\rangle \in \B{M\dual}$ with $\birth' \leq \birth$.  By \cref{Prop:BarcodesUnderDuality}, this map in turn induces an injection \(\iota:\B{N} \hookrightarrow \B{M} \), which is exactly the canonical injection. Since $\iota$ maps each $\langle \birth, \death\rangle \in \B{N}$ to an interval $\langle \birth, \death'\rangle \in \B{M}$ with $\death\leq \death'$,
the result follows.
\end{proof}

\begin{remark}\label{MVJ_Proof_Of_Struct_Thm}
After we announced our results, Primo\v z \v Skraba and Mikael Vejdemo-Johansson shared with us a nice alternate proof of \cref{Prop:StructPropForQuotientsAndSubmodules}\,(ii) for the special case of finitely generated $\Nbb$-indexed persistence modules.  By duality, this yields \cref{Prop:StructPropForQuotientsAndSubmodules}\,(i) for the special case as well.  

The proof uses the graded Smith Normal Form algorithm for matrices with homogeneous $k[t]$-coefficients, as discussed in  \cite{zomorodian2005computing,skraba2013persistence}.  Here is a brief outline of the argument: Given an epimorphism $M\twoheadrightarrow N$ of finitely generated $\Nbb$-graded persistence modules and a presentation matrix $P_M$ for $M$ in graded Smith Normal Form, we can extend $P_M$ to a presentation matrix $P_N$ for $N$ by adding more columns; see \cite%
{skraba2013persistence}.  By considering the behavior of the graded Smith Normal Form algorithm on $P_N$, we can show that the barcode of $N$ is obtained from the barcode of $M$ by moving the right endpoints of intervals in the barcode of $M$ to the left.  In fact, for this it suffices to consider just the column operations performed by the algorithm, since these are enough to reduce $P_N$ to a form from which we can read off the barcode of $N$ \cite{zomorodian2005computing}.
The special case of \cref{Prop:StructPropForQuotientsAndSubmodules}\,(ii) follows readily.
\end{remark}

\section{Induced Matchings of Barcodes}\label{Sec:InducedMatchingDef}
We now define the map $\barc{}$ sending each morphism $f:M \to N$ of pointwise finite dimensional persistence modules to a matching $\barc f: \B{M}\matching \B{N}$.  

We first define the map for monomorphisms and epimorphisms.  For $j:M\hookrightarrow N$ a monomorphism, we define $\barc{j}:\B{M}\matching \B{N}$ 
to be the canonical injection $\B{M}\hookrightarrow \B{N}$ of \cref{Prop:StructPropForQuotientsAndSubmodules}, considered as a matching.  Dually, for $q:M\twoheadrightarrow N$ an epimorphism, we define $\barc{q}:\B{M}\matching \B{N}$ to be the reverse of the canonical injection $\B{M}\hookleftarrow \B{N}$.
 
Now consider an arbitrary morphism $f:M\to N$ of \pfd persistence modules.  $f$~factors (canonically) as a composition of morphisms 
\[M\stackrel{q_f}\twoheadrightarrow\im f\stackrel{j_f}\hookrightarrow N,\] 
where $q_f$ is an epimorphism and $j_f$ is a monomorphism.  We define 
\[\barc{f}=\barc{j_f}\circ \barc{q_f}.\]

\begin{remark}
To build intuition for our definition of $\barc{f}$, the reader may find it helpful to consider the definition in the generic case that $|{\intervalSub{\birth}{\wildcard}}_M|\leq 1$ and $|{\intervalSub{\wildcard}{\death}}_N|\leq 1$ for all $\birth,\death\in \E$.  In this case, the definition becomes especially simple: for each interval $I=\langle \birth, \death\rangle \in \B{\im{f}}$, there is a unique interval $I_M=\langle \birth, \death'\rangle \in \B{M}$ and a unique interval $I_N=\langle \birth', \death\rangle \in \B{N}$.  We have $\barc{f}=\{(I_M,I_N)\mid I\in \B{\im f}\}$.
\end{remark}

\begin{figure}[hbt]
\centering{\includegraphics[scale=1.]{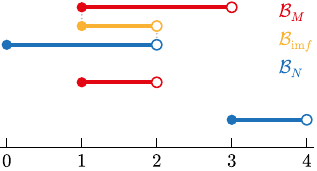}}
{\caption{Barcodes of $M$, $\im f$, and $N$, as in \cref{Ex:barcodes}. Matched intervals are grouped together.}
\label{Fig:barcode-mike}}
\end{figure}
\begin{example}
\label{Ex:barcodes}
Let $M=C[1,2)\oplus C[1,3)$, $N=C[0,2)\oplus C[3,4)$.  Let $f:M\to N$ be a morphism which maps the summand $C[1,2)$ injectively into the summand $C[0,2)$ and maps the summand $C[1,3)$ to 0.  Then $\im{f}\cong C[1,2)$.  The barcodes $\B{M}$, $\B{\im f}$, and $\B{N}$ are plotted in \cref{Fig:barcode-mike}.  We have $\barc{f}=\{([1,3),[0,2))\}$.
\end{example}

\subsection{Properties of Induced Matchings}
Several key facts about the induced matchings $\barc{f}$ follow almost immediately from the definition.  We record them here.

\begin{proposition}\label{Prop:IMT_Part_1}
Let $f:M\to N$ be a morphism of \pfd persistence modules and suppose $\barc{f}\interval{ \birth}{\death }=\interval{ \birth'}{\death' }$.  Then \[\birth' \leq \birth<\death'\leq \death.\]
\end{proposition}

\begin{figure}[hbt]
\centerline{\includegraphics[scale=1.]{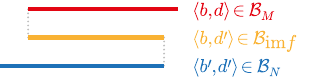}}
{\caption{A pair of intervals matched by $\protect\barc{f}$, together with the corresponding interval in $\protect\B{\im(f)}$.}
\label{Fig:barcode-mike}}
\end{figure}

\begin{proof}
By the definition of $\barc{f}$, we have $\barc{q_f}\interval{ \birth}{\death }=\interval{ \birth}{\death'}$ and $\barc{j_f}\interval{ \birth}{\death'}=\interval{ \birth'}{\death' }$.  By \cref{Prop:StructPropForQuotientsAndSubmodules},
\[\birth' \leq \birth <\death'\leq \death,\]
where the middle inequality holds because $\interval{\birth}{\death'}$ is an interval.
\end{proof}

Our induced matching theorem, presented in \cref{Sec:TheoremStatements}, is a refinement of \cref{Prop:IMT_Part_1}.

\begin{proposition}\label{Prop:ImageBarcodeFromInducedMatching}
For any morphism $f:M\to N$ of \pfd persistence modules, $\barc{f}$ is completely determined by $\B M$, $\B N$, and $\B{\im{f}}$.  Conversely, $\barc{f}$ completely determines these three barcodes.  
\end{proposition}

\begin{proof}

It is clear from the definition that $\barc{f}$ depends only on $\B M$, $\B N$, and $\B{\im{f}}$.  $\barc{f}$ is a matching between $\B M$ and $\B N$,  so by definition $\barc{f}$ determines $\B M$ and $\B N$.  To prove the converse then, it is enough to check that \[\B{\im{f}}=\{I \cap J \mid (I,J) \in \barc{f}\},\]
where we interpret the right hand side as the representation of a multiset in the obvious way.  By the definition of $\barc{f}$, the map \(\barc{f} \to \B{\im f}\) which sends each matched pair $(I,J)\in \barc{f}$ to $\barc{q_f}(I)\in \B{\im f}$ is a bijection.  To obtain the result, we note that if $\barc{f}\interval{\birth}{\death}=\interval{\birth'}{\death'}$, then \[\barc{q_f}\interval{\birth}{\death}=\interval{\birth}{\death'}=  \interval{\birth}{\death} \cap \interval{\birth'}{\death'},\]
where the second equality follows from \cref{Prop:IMT_Part_1}.
\end{proof}

\paragraph{Induced Matchings and Duality}
We next observe that the matching $\barc{f\dual}$ is, up to indexing considerations, the reverse of the matching $\barc{f}$.  %

For $\mathcal D$ a barcode, recall our definition of $\mathcal D\dual$ from \cref{Sec:dualization}.  A matching $\sigma: \mathcal C\matching \mathcal D$ between barcodes canonically induces a matching $\sigma':\mathcal C\dual\matching \mathcal D\dual$, and hence a matching $\sigma\dual:\mathcal D\dual\matching \mathcal C\dual$ obtained simply by reversing $\sigma'$.  With these definitions, $(\wildcard)\dual$ is a contravariant endofunctor on the full subcategory of $\MatchingCat$ whose objects are barcodes.

The following is a counterpart of \cref{Prop:BarcodesUnderDuality} for induced matchings:

\begin{proposition}\label{Prop:InducedMatchingsUnderDuality}
If $f:M\to N$ is a morphism of \pfd persistence modules, then $\barc{f\dual}=\barc{f}\dual$.
\end{proposition}

\begin{proof}
If $q$ is an epimorphism of persistence modules and $j$ is a monomorphism of persistence modules, then $q\dual$ is a monomorphism of persistence modules and $j\dual$ is an epimorphism of persistence modules.
Thus, for $f=j_f\circ q_f$ we have
\[\barc{f\dual}=\barc{(j_f\circ q_f)\dual}=\barc{q_f\dual \circ j_f\dual}=\barc{q_f\dual} \circ \barc{j_f\dual}.\]
It is easy to check that if $g$ is an epimorphism or a monomorphism, then $\barc{g}\dual=\barc{g\dual}$.  Therefore,
\[\barc{q_f\dual} \circ \barc{j_f\dual}=\barc{q_f}\dual \circ \barc{j_f}\dual=(\barc{ j_f}\circ \barc{q_f })\dual=\barc{f}\dual.\]
We thus have $\barc{f\dual}=\barc{f}\dual$ as desired.  
\end{proof}

\subsection{Partial Functoriality of Induced Matchings}\label{Sec:LimitedFunctoriality}
Let \[\B{}:\obj(\kvect^\RCat)\to \obj(\MatchingCat)\] denote the map sending a \pfd persistence module $M$ to its barcode $\B{M}$.
Given that \[\barc{}:\hom(\kvect^\RCat)\to \hom(\MatchingCat)\] maps each morphism $f:M\to N$ of persistence modules to a matching between $\B M$ and $\B N$, one might hope that the pair $(\B{},\barc{})$ defines a functor $\kvect\to \MatchingCat$.  However, as the next example shows, $\barc{}$ does not respect composition of morphisms, so is not functorial.
\begin{example}
Consider persistence modules 
\begin{align*}
L&=C[3,\infty)\oplus C[4,\infty),\\
M&=C[2,\infty)\oplus C[1,\infty),\\
N&=C[0,\infty).
\end{align*}
Define $f:L\to M$ by $f(s,t)=(s,0)$; define $g:M\to N$ by $g(s,t)=t$.  We have $\im f\cong C[3,\infty)$ and $\im g\cong C[1,\infty)$.

$\barc{f}$ matches the interval $[3,\infty) \in \B{L}$ to $[1,\infty) \in \B{M}$, and leaves $[4,\infty) \in \B{L}$ and $[2,\infty) \in \B{M}$ unmatched.  $\barc{g}$ matches $[1,\infty) \in \B{M}$ to $[0,\infty) \in \B{N}$ and leaves $[2,\infty) \in \B{M}$ unmatched.  Thus  $\barc{g}\circ \barc{f}$ matches $[3,\infty) \in \B{L}$ to $[0,\infty) \in \B{N}$ and leaves $[4,\infty) \in \B{L}$ unmatched.  
On the other hand $g\circ f=0$, so $\barc{g\circ f}$ is the trivial matching (i.e., $\barc{g\circ f}=\emptyset$).
\end{example}

Though $\barc{}$ is not functorial in general, we have the following partial functoriality result for $\barc{}$:
\begin{proposition}\label{Prop:LimitedFunctoriality}
$\barc{}$ is functorial on the two subcategories of $\kvect^\RCat$ whose morphisms are, respectively, the monomorphism and epimorphisms.
\end{proposition}

\begin{proof}
It is clear from the definition that $\barc{}$ sends identity morphisms to identity morphisms.  If $j:M\to N$ is a monomorphism of \pfd persistence modules, then by definition $\barc{j}$ factors as the disjoint union of the canonical injections ${\intervalSub{\wildcard}{\death}}_M\hookrightarrow {\intervalSub{\wildcard}{\death}}_N$ indexed by endpoints $\death\in \E$.  It thus follows by \cref{Lem:Functoriality} that if $j_1:L\to M$ and $j_2:M\to N$ are monomorphisms of \pfd persistence modules, then $\barc{j_2}\circ \barc{j_1}=\barc {j_2\circ j_1}.$
Thus, $\barc{}$ is functorial on the subcategory of monomorphisms in $\kvect^\RCat$.  

Using essentially the same argument, together with the fact that the operation of reversing a matching is functorial,
we find that, dually, $\barc{}$ is functorial on the subcategory of epimorphisms in $\kvect^\RCat$.
\end{proof}

\begin{example}
Given \pfd persistence modules $M,N,P,Q$ and morphisms $f:M\to N$, $g:P\to Q$ we can define the direct sum $f\oplus g:M\oplus P\to N\oplus Q$ in the obvious way.  In view of \cref{Prop:LimitedFunctoriality}, one might hope that when $f$ and $g$ are monomorphisms, $\barc{f\oplus g}=\barc{f}\coprod \barc{g}$.  However, this equality does not hold in general.  For example, we can take $M=N=C[1,2)$, $P=0$, and $Q=C[0,2)$, with $f:M\to N$ the identity map.  Then $\barc{f}\coprod \barc{g}=\{([1,2),[1,2))\}$ but $\barc{f\oplus g}=\{([1,2),[0,2))\}.$

\end{example}

\newcommand\barcodeFunctor{S}

\paragraph{Non-Functoriality of Persistence Barcodes}
It is natural to ask whether there exists some other, fully functorial definition of a map sending each morphism $M \to N$ of \pfd persistence modules to a matching $\B{M}\matching\B{N}$.  \cref{Prop:NoFunctoriality} below makes clear that the answer is no.

\begin{lemma}\label{Lem:VectorSpaceNonFunctoriality}
There exists no functor $\kvect\to \MatchingCat$ sending each vector space of dimension~$d$ to a set of cardinality~$d$.
\end{lemma}

\begin{proof}
Assume for a contradiction that such a functor $\barcodeFunctor$ exists.
We first show that $\barcodeFunctor$ must send each linear map $f:V\to W$ of rank $r$ to a matching $\barcodeFunctor(f)$ of cardinality $|\barcodeFunctor(f)|=r$.  We begin with the case of $f$ an injection.  In this case, $f$ has a left inverse $g:W\to V$.  We have  \[\barcodeFunctor(g)\circ \barcodeFunctor(f)=\barcodeFunctor(g\circ f)=\barcodeFunctor(\id_V),\] so \[r=|\barcodeFunctor(\id_V)|=|\barcodeFunctor(g)\circ \barcodeFunctor(f)|\leq |\barcodeFunctor(f)| \leq \dim V = r,\]
which gives the result.  By an analogous argument, the result also holds when $f$ is a surjection.  

For the general case, consider $f$ as a composition of maps $V\stackrel{q_f}\twoheadrightarrow\im f\stackrel{j_f}\hookrightarrow W$.  Since $\rank q_f = \rank j_f = \rank f = r$, we have $|\barcodeFunctor(q_f)|=|\barcodeFunctor(j_f)| = r$. Since $|\barcodeFunctor(\im f)|=r$ as well, this implies that both $\barcodeFunctor(q_f)$ and $\barcodeFunctor(j_f)$ match every element of $\barcodeFunctor(\im f)$.
It follows that $|\barcodeFunctor(f)|=r$.

Now let $i_1,i_2,i_3:\K\to\K^2$ be injective linear maps with left inverses $p_1,p_2,p_3:\K^2\to\K$, respectively, such that 
$p_1 \circ i_2 = p_2 \circ i_1 = 0$ and $p_1 \circ i_3 = p_2 \circ i_3 = \id_K$.
For instance,%
\begin{align*}
i_1&: x \mapsto (x,0), & i_2&: x \mapsto (0,x), &i_3&: x \mapsto (x,x), \\
p_1&: (x,y) \mapsto x, & p_2&: (x,y) \mapsto y, &p_3&: (x,y) \mapsto x.
\end{align*}
Each of these six maps is of rank 1 and thus matches exactly one element of $\barcodeFunctor(K^2)$.  Write $\barcodeFunctor(\K^2)=\{a,b\}$ and assume without loss of generality that $\barcodeFunctor(i_1)$ matches $a$.  Since $\barcodeFunctor(p_1) \circ \barcodeFunctor(i_1) = \barcodeFunctor(p_1 \circ i_1) = {\barcodeFunctor(\id_K)} = \id_{\barcodeFunctor(K)}$, $\barcodeFunctor(p_1)$ must also match $a$.   Moreover, $\barcodeFunctor(i_2)$ must match $b$; otherwise, $\barcodeFunctor(i_2)$ would have to match $a$, and we would have $\barcodeFunctor(p_1) \circ \barcodeFunctor(i_2) \neq\emptyset$ and $\barcodeFunctor(p_1 \circ i_2)=\barcodeFunctor(0)=\emptyset$, violating the functoriality of $\barcodeFunctor$.  By an analogous argument, $\barcodeFunctor(p_2)$ must match $b$ as well.

If $i_3$~matches~$a$, then $\barcodeFunctor(p_2) \circ \barcodeFunctor(i_3) = \emptyset$. But $\barcodeFunctor(p_2 \circ i_3) = \barcodeFunctor(\id_K) \neq \emptyset$, violating functoriality.  On the other hand, if $i_3$~matches~$b$, then $\barcodeFunctor(p_1) \circ \barcodeFunctor(i_3) = \emptyset$ but $ \barcodeFunctor(p_1 \circ i_3) \ne \emptyset$, again violating functoriality.  We conclude that the functor $\barcodeFunctor$ cannot exist.
\end{proof}

\begin{proposition}\label{Prop:NoFunctoriality}
There exists no functor $\kvect^\RCat \to \MatchingCat$ sending each persistence module to its barcode.
\end{proposition}

\begin{proof}
Assume for a contradiction that there exists a functor
$B:\kvect^\RCat\to \MatchingCat$ sending each persistence module to its barcode.
Consider the functor $P: \kvect \to \kvect^\RCat$ that sends a vector space~$V$ to the persistence module $M$ with $M_0=V$ and $M_t=0$ for $t\neq0$.
The functor $B \circ P:\kvect\to \MatchingCat$  then sends each vector space of dimension $d$ to a barcode of cardinality $d$.  But by \cref{Lem:VectorSpaceNonFunctoriality}, such a functor cannot exist.  We conclude that there cannot exist a functor $\kvect^\RCat\to \MatchingCat$ sending each persistence module to its barcode.
\end{proof}

\section{The Induced Matching Theorem}\label{Sec:TheoremStatements}
We now come to our main result, the induced matching theorem.

For $\epsilon\geq 0$, we say a persistence module $M$ is \emph{$\epsilon$-trivial} if the transition morphism $\varphi_M^\epsilon:M\to M(\epsilon)$ is the zero morphism.  
Recall that in \cref{Sec:BottleneckDistance}, 
for a barcode $\D$ and $\epsilon\geq 0$ we defined 
\[\D^{\epsilon} = \{ \interval{\birth}{\death } \in \D \mid \birth + \epsilon < \death \}=\{I\in \D\mid [t,t+\epsilon]\subseteq I \textup{ for some }t\in \R\}.\]
\begin{theorem}[Induced Matching Theorem]\label{Thm:InducedMatching}
Let $f:M\to N$ be a morphism of pointwise finite dimensional persistence modules and suppose $\barc{f}\interval{ \birth}{\death }=\interval{ \birth'}{\death' }$.  Then
\begin{enumerate}[(i)]
\item $\birth' \leq \birth<\death'\leq \death$.
\item If $\coker f$ is $\epsilon$-trivial, then $\B{N}^\epsilon\subseteq \im \barc{f}$ and 
\[\birth' \leq \birth \leq \birth'+\epsilon.\]
\item Dually, if $\ker f$ is $\epsilon$-trivial, then $\B{M}^\epsilon\subseteq \coim \barc{f}$ and 
\[\death-\epsilon \leq \death'\leq \death.\]
\end{enumerate}
\end{theorem}

\begin{figure}[hbt]
\centerline{\includegraphics[scale=1.]{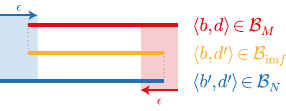}}
{\caption{The relationship between a pair of intervals matched by $\protect\barc{f}$, as described by the induced matching theorem.}

\label{Fig:induced-matching-close-eps}}
\end{figure}

\begin{proof}[Proof of \cref{Thm:InducedMatching}]
\cref{Thm:InducedMatching}\,(i) is exactly \cref{Prop:IMT_Part_1}.

To prove \cref{Thm:InducedMatching}\,(ii), let  $N^\epsilon$ be the submodule of $N$ given by $N^\epsilon_t=\im \varphi_M(t-\epsilon,t)$ for all $t\in \R$.
Note that \[\B{N^\epsilon}=\{\interval{ c+\epsilon}{e } \mid \interval{ c}{e } \in \B{N}^\epsilon\}.\]  
Put informally, we obtain the barcode $\B{N^\epsilon}$ from $\B{N}$ by first removing the intervals of length less than $\epsilon$ from $\B{N}$, and then moving the left endpoint of each remaining interval to the right by $\epsilon$.
Our proof strategy will be to sandwich $\B{\im f}$ between $\B{N^\epsilon}$ and $\B{N}$ via canonical injections.

Let $j_\epsilon:N^\epsilon \hookrightarrow N$ denote the
inclusion.  We see from the definition of $\barc{j_\epsilon}$ that $\im\barc{j_\epsilon}=\B{N}^\epsilon$, and if $\interval{c}{e} \in \B{N}^\epsilon$, then 
\[\barc{j_\epsilon}\interval{c+\epsilon}{e}=\interval{c}{e}.\]

\begin{figure}[hbt]
\centerline{\includegraphics[scale=1.]{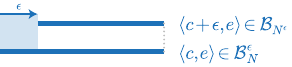}}
{\caption{The relationship between a pair of intervals matched by $\protect\barc{j_\epsilon}:\protect\B{N^\epsilon} \hookrightarrow \protect\B{N}$.}
\label{Fig:Sandwich}}
\end{figure}

Now assume that $\coker f$ is $\epsilon$-trivial.  Then $N^\epsilon \subseteq \im{f}$.  (To see this, note that for all $t\in \R$, $
\varphi_{\coker f}(t-\epsilon,t)=0$ if and only if $\im \varphi_{N}(t-\epsilon,t)\subseteq \im f_t$.)  Thus the following diagram commutes, where each map is the inclusion:
\[
\begin{tikzcd}[row sep=3ex, column sep=4.5ex]
&
N 
\\
{N^\epsilon}
\arrow[hook]{ur}{{j_\epsilon}} 
\arrow[hook,swap]{dr}{{j}}
&
\\
&
{\im f} 
\arrow[hook,swap]{uu}{{j_f}} 
\end{tikzcd}
\]
By \cref{Prop:LimitedFunctoriality}, we have \(\barc{j_\epsilon}=\barc{j_f}\circ \barc{j}.\)
Moreover, by definition we have that \(\barc{f}=\barc{j_f}\circ \barc{q_f},\)
so the following diagram of matchings commutes:
\[
\begin{tikzcd}[row sep=3ex, column sep=3ex]
&
\B N 
\\
\B{N^\epsilon}
\arrow[hook,strike thru]{ur}{\barc{j_\epsilon}} 
\arrow[hook,strike thru,swap]{dr}{\barc{j}}
&& 
\B M 
\arrow[strike thru,swap]{ul}{\barc{f}} 
\arrow[two heads,strike thru]{dl}{\barc{q_f}} 
\\
&
\B{\im f} 
\arrow[hook,strike thru,swap]{uu}{\barc{j_f}} 
\end{tikzcd}
\]
By the commutativity of the left triangle, $\B{N}^\epsilon=\im\barc{j_\epsilon}\subseteq\im\barc{j_f}$.  By our definition of induced matchings,
we have that 
$\im\barc{j_f}=\im\barc{f}$,
so 
\[\B{N}^\epsilon\subseteq\im\barc{f}\] as claimed.

\begin{figure}[hbt]
\centerline{\includegraphics[scale=1.]{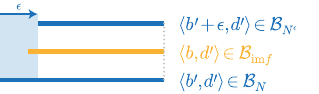}}
{\caption{To show that $b \leq b'+\epsilon$, we sandwich $b$ between $b'$ and $b'+\epsilon$, using that $\protect\barc{j_\epsilon}=\protect\barc{j_f}\circ \protect\barc{j}$.}
\label{Fig:Sandwich}}
\end{figure}

To finish the proof of \cref{Thm:InducedMatching}\,(ii), we need to show that for $b,$ $b'$ as in the statement of the result, $\birth \leq \birth'+\epsilon$.  This follows from the commutativity of the left triangle as well.  To see this, recall that 
$\barc{j_f}\interval{ \birth}{\death' }=\interval{ \birth'}{\death' }$.
If $\interval{ \birth'}{\death' } \in \B{N}^\epsilon$, then we also have that $\barc{j_\epsilon}\interval{\birth'+\epsilon}{\death'}=\interval{\birth'}{\death'}$, so
$\barc{j}\interval{ \birth'+\epsilon}{\death' }=\interval{ \birth}{\death' }$ by commutativity of the triangle; see \cref{Fig:Sandwich}.   
By \cref{Prop:StructPropForQuotientsAndSubmodules}\,(i), we have that \[\birth \leq \birth'+\epsilon.\]
If, on the other hand, $\interval{ \birth'}{\death' } \not\in \B{N}^\epsilon$, then $\death' \leq \birth'+\epsilon$.  Since $\birth< \death'$, we again have that \[\birth \leq \birth'+\epsilon.\]

It remains to prove \cref{Thm:InducedMatching}\,(iii). %
Our proof of \cref{Thm:InducedMatching}\,(ii) dualizes readily.  Alternatively, \cref{Thm:InducedMatching}\,(iii) follows directly from \cref{Thm:InducedMatching}\,(ii) via a straightforward duality argument. 
This latter argument requires the use of \cref{Prop:InducedMatchingsUnderDuality} and \cref{Lem:KernelAndCokernelTriviality,} below, whose proof we omit.  %
\end{proof}

\begin{lemma}\label{Lem:KernelAndCokernelTriviality}
For $f:M\to N$ a morphism of \pfd persistence modules, $\ker f$ is $\epsilon$-trivial if and only if\/ $\coker f\dual$ is $\epsilon$-trivial.  
\end{lemma}

\subsection{Algebraic Stability via Induced Matchings}\label{Sec:ASP}
We now establish an explicit form of the algebraic stability theorem for pointwise finite dimensional persistence modules as a corollary of the induced matching theorem.  

\begin{lemma}\label{Lem:DeltaTrivialityAndInterleavings}
If $f:M\to \shift N\delta$ is a $\delta$-interleaving morphism, then $\ker{f}$ and $\coker{f}$ are both $2\delta$-trivial.
\end{lemma}
\begin{proof} This follows from the definition of a $\delta$-interleaving.  
\end{proof}
For any persistence module $N$ and $\delta\in [0,\infty)$, we have a bijection \[r_\delta:\B{\shift{N}{\delta}}\to\B{N}\] such that $r_\delta\interval{ \birth}{\death}=\interval{ \birth+\delta}{\death+\delta}$ for each interval $\interval{ \birth}{\death}\in \B{\shift N \delta}$.
\begin{theorem}[Explicit Formulation of Algebraic Stability]\label{Thm:pfdASP} 
If $M, N$ are pointwise finite dimensional persistence modules and $f:M\to \shift N \delta$ is a $\delta$-interleaving morphism, then $r_\delta \circ \barc{f}:\B{M}\matching \B{N}$ is a $\delta$-matching.  
In particular, \[d_B(\B M,\B N)\leq d_I(M,N).\]
\end{theorem}

\begin{proof}  By \cref{Lem:DeltaTrivialityAndInterleavings}, 
$\ker{f}$ and $\coker{f}$ are both $2\delta$-trivial.
The Induced Matching \cref{Thm:InducedMatching} thus tells us that 
$\B{M}^{2\delta}\subseteq \coim \barc{f}$, 
$\B{N(\delta)}^{2\delta}\subseteq \im \barc{f}$,
and
each matched pair of bars  
$\barc{f}\interval{ \birth}{\death }=\interval{ \birth'}{\death' }$ 
satisfies
\[\birth' \leq \birth \leq \birth'+{2\delta}
\text{\quad and\quad} 
\death-{2\delta} \leq \death'\leq \death.\]
It now follows from the definition of $r_\delta$ that $r_\delta \circ \barc{f}$ is a $\delta$-matching.
\end{proof}

\begin{figure}[hbt]
\centerline{\includegraphics[scale=1.]{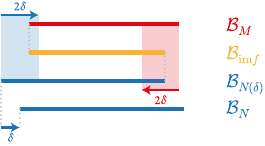}}
{
\caption{
A pair of intervals in the $\delta$-matching $r_\delta \circ \protect\barc{f}:\protect\B{M}\protect\matching \protect\B{N}$ of \cref{Thm:pfdASP}, together with the corresponding intervals in $\protect \B{\im f}$ and $\protect \B{N(\delta)}$.}
\label{Fig:induced-matching-AST}}
\end{figure}

\begin{remark}
For $M$ and $N$ $\delta$-interleaved persistence modules, our proof of the algebraic stability theorem gives a construction of a $\delta$-matching between $\B M$ and $\B N$ depending only on a single $\delta$-interleaving morphism $f:M\to \shift{N}{\delta}$; to construct the $\delta$-matching, one does not need an explicit choice of morphism $g:N\to \shift{M}{\delta}$ such that 
\begin{align*}
\shift g \delta \circ f&=\translation M {2\delta},\\
\shift f \delta \circ g&=\translation N {2\delta}. 
\end{align*}  
This is in contrast to the proof strategy via interpolation used in \cite{chazal2009proximity} and \cite{chazal2012structure}: The construction of a matching given in those proofs depends in an essential way on an explicit choice of both $f$ and $g$.
\end{remark}

\subsection{Single-Morphism Characterization of the Interleaving Relation}
As an another application of \cref{Thm:InducedMatching}, we have the following:

\begin{corollary}\label{OneSidedCharacterization}
Two \pfd persistence modules $M$ and $N$ are $\delta$-interleaved if and only if there exists a morphism $f:M\to N(\delta)$ with $\ker f$ and $\coker f$ both $2\delta$-trivial.
\end{corollary}

\begin{proof}
The forward direction is \cref{Lem:DeltaTrivialityAndInterleavings}.  The converse direction follows from the induced matching theorem and the converse algebraic stability theorem~\ref{Thm:ConverseStabilityForDecomposableMods}.
\end{proof}

\begin{remark}
Our definitions of $\delta$-interleavings and $\delta$-trivial persistence modules extend readily to multidimensional persistence modules; see \cite{lesnick2013theory}.  Given this, it is natural to ask whether \cref{OneSidedCharacterization} generalizes to the multidimensional setting.

It is easy to check that the forward direction of the theorem (i.e.,  \cref{Lem:DeltaTrivialityAndInterleavings}) does generalize to multidimensional persistence modules, with the same easy proof.  Conversely, by considering the canonical factorization of a morphism though its image, it is not hard to check that the converse direction of \cref{OneSidedCharacterization} holds for multidimensional persistence modules, up to a factor of 2.  Perhaps surprisingly, the next example shows that 
showing that this factor of 2 is tight.  Thus, while \cref{OneSidedCharacterization} does adapt to the multidimensional setting, the general result is not as strong as the result in one dimension.
\end{remark}

\begin{example}
In this example, we freely use the notation and terminology for multidimensional persistence modules established in \cite{lesnick2013theory}.  First, for $\delta\geq 0$ we define a 2-D persistence module $P$ to be \emph{$\delta$-trivial} if $\varphi_P(\delta):P\to P(\delta)$, as defined in \cite{lesnick2013theory}, is the trivial morphism.  

We will describe 2-D persistence modules $M$ and $N$ and a morphism $f:M\to N(1)$ with 2-trivial kernel and cokernel, such that $M, N$ are  $\delta$-interleaved if and only if $\delta\geq 2$.  

We specify $M$ and $N$ via presentations.  We take \[M\cong \langle (a, (3,2)), (b, (2,3)) \mid x^2y^3 a-x^3y^2b\rangle.\]  That is, $M$ has a generator $a$ at grade $(3,2)$, a generator $b$ at $(2,3)$, and a relation $x^2y^3 a-x^3y^2b$ at grade $(5,5)$.  
We take \[N\cong \langle (c, (2,1)), (d, (1,2)) \mid yc-xd\rangle.\]  %
Thus \[N(1)\cong\langle (c', (1,0)), (d', (0,1)) \mid yc'-xd'\rangle.\]   %
We define $f:M\to N(1)$ by taking $f(a)=x^2y^2c'$, $f(b)=x^2y^2d'$. 
It is easy to check that $f$ has 2-trivial kernel and cokernel, and that $M$ and $N$ are $\delta$-interleaved if and only if $\delta\geq 2$.
\end{example}

\section{Algebraic Stability for q-Tame Persistence Modules}\label{Sec:QTameASP}
In the remaining two sections, we show that the isometry theorem for q-tame persistence modules, as presented in \cite[Theorem 4.11]{chazal2012structure}, follows readily from the isometry theorem for \pfd persistence modules.  

First, we derive the algebraic stability theorem for q-tame persistence modules, as stated in \cite[Theorem 4.20]{chazal2012structure}, from the algebraic stability theorem for \pfd persistence modules, \cref{Thm:pfdASP}.  The last section treats the converse algebraic stability theorem in both the \pfd and q-tame settings.    %

Neither \cite[Theorem 4.20]{chazal2012structure} nor our \cref{Thm:pfdASP} immediately implies the other.  On the one hand, \cite[Theorem 4.20]{chazal2012structure} applies to q-tame persistence modules, whereas \cref{Thm:pfdASP} only applies to \pfd persistence modules.  On the other hand, \cref{Thm:pfdASP} is slightly sharper than the restriction of \cite[Theorem 4.20]{chazal2012structure} to \pfd persistence modules, because \cite[Theorem 4.20]{chazal2012structure} concerns undecorated barcodes (see below), and because as noted in \cref{delta_matching_remark}, the definition of $\delta$-matching used in the present paper is stronger than the one used in \cite{chazal2012structure}. 

\subsection{Undecorated Barcodes of q-Tame Persistence Modules}
Recall %
 that we say a persistence module $M$ is q-tame if $\transition M s t$ has finite rank whenever $s < t$.  A q-tame persistence module does not necessarily decompose into interval summands \cite[Section 4.5]{chazal2012structure}, and thus our definition of barcode does not extend to q-tame persistence modules.

However, it is observed in \cite{chazal2014observable} that if we are willing to use a slightly coarser notion of barcode, then we can define the barcode of a q-tame persistence module $M$ by way of the interval summand decomposition of an approximation of $M$.  We recall the definition here.  See \cite{chazal2014observable} for a thorough study of the definition, including an interpretation in terms of quotient categories.

\paragraph{Undecorated Barcodes}
Adapting a definition of \cite{chazal2012structure}, we define an \emph{undecorated barcode} to be a barcode for which every interval is open.  In what follows, %
we will sometimes refer to barcodes as \emph{decorated barcodes}, to emphasize the distinction between general barcodes and the special case of undecorated barcodes.  

There is an obvious map $\undecorate$ from the set of decorated barcodes to the set of undecorated barcodes, which removes all intervals of the form $[t,t]$ from a barcode and forgets the decorations of the endpoints of all other intervals in the barcode.  For example, %
we have \[\undecorate({\{[0,1),[0,2],(-\infty,\infty),[0,0]\}})=\big\{(0,1),(0,2),(-\infty,\infty)\big\}.\]  
If $M$ is an interval-decomposable persistence module, we define the undecorated barcode $\Bopen M$ of $M$ to be $\undecorate{(\B{M})}$.

\paragraph{Undecorated Barcodes of q-Tame persistence modules}
As in the proof of \cref{Thm:InducedMatching}, for $M$ a persistence module and $\epsilon\geq 0$, let $M^\epsilon$ denote the submodule of $M$ defined by $M^\epsilon_t=\im \varphi_M(t-\epsilon,t)$ for all $t\in \R$. 

For $M$ any persistence module, define $\rad M$, the \emph{radical} of $M$, as \[\rad M=\bigcup_{\epsilon>0}M^\epsilon.\]  Note that $\rad M$ is a submodule of $M$.  It is easy to see that $d_I(M,\rad M)=0$; in this sense, $\rad M$ is an infinitesimally close approximation of $M$.  %

\begin{theorem}[Structure of Radical q-Tame Persistence Modules \protect{\cite[Corollary~3.6]{chazal2014observable}}]\label{Thm:StructForQTameRads}  For $M$ any q-tame persistence module, $\rad M$ is interval decomposable.
\end{theorem}
See also \cite[Theorem~3.2]{chazal2014observable} for a simultaneous generalization of both \cref{Thm:StructForQTameRads} and the structure theorem for \pfd persistence modules \cite{crawley2012decomposition}.

We state the following easy result without proof:
\begin{proposition}\label{UndecoratedBarcodeWellDefined}
For any interval-decomposable persistence module $M$, $\undecorate_{\rad M}=\undecorate_{M}$.  
\end{proposition}

For $M$ q-tame, we define $\Bopen M$, the undecorated barcode of $M$, by taking $\Bopen M=\undecorate_{\rad M}$.  In view of \cref{UndecoratedBarcodeWellDefined}, for $M$ a persistence module  that is both q-tame and interval-decomposable, this definition coincides with the definition of $\Bopen M$ given above.

\begin{remark}
\cite{chazal2012structure} defines the undecorated barcode of a q-tame persistence module by first defining the decorated barcode of a q-tame persistence module $M$ using the rectangle measure formalism introduced in that paper, and then taking the undecorated barcode of $M$ to be the image of the decorated barcode of $M$ under $\undecorate$; it follows from the results of \cite{chazal2012structure} that the two definitions of the undecorated barcode of a q-tame persistence module are equivalent.
\end{remark}

\subsection{Algebraic Stability for q-Tame Persistence Modules}
Using the algebraic stability theorem for \pfd persistence modules, we now give our alternative proof of the algebraic stability theorem for q-tame persistence modules.

\begin{theorem}[{Algebraic Stability for q-Tame Modules \cite[Theorem 4.20]{chazal2012structure}}]\label{Thm:ASP_ForQTame}
If $M$ and $N$ are q-tame and are $(\delta+\epsilon)$-interleaved for all $\epsilon>0$, then there is a $\delta$-matching between $\Bopen M$ and $\Bopen N$. 
In particular, \[d_B(\Bopen M,\Bopen N)\leq d_I(M,N).\]
\end{theorem}

To prepare for the proof of the theorem, we make some easy technical observations, leaving the proofs to the reader:
\begin{lemma}\label{Lem:ModuleApproximations}
Let $M$ be a q-tame persistence module.
\begin{enumerate}[(i)]
\item $M$ and $M^\epsilon$ are $\epsilon$-interleaved.  
\item For any $\epsilon>0$, $M^\epsilon$ is \pfd  
\item There is a canonical $\epsilon$-matching $\B{\rad M} \matching \B{M^\epsilon}$.
\end{enumerate}
\end{lemma}

\Cref{Lem:ModuleApproximations}\,(i,ii) implies in particular that a q-tame persistence module can be approximated arbitrarily well by a \pfd persistence module.

\begin{proof}[Proof of \cref{Thm:ASP_ForQTame}]
$M$ and $N$ are $(\delta+\epsilon)$-interleaved for all $\epsilon>0$, so by \cref{Lem:ModuleApproximations}\,(i,ii) and the triangle inequality for interleavings (\cref{EasyInterleavingRemark}), $M^\epsilon$ and $N^\epsilon$ are \pfd and are $(\delta+3\epsilon)$-interleaved.  By \cref{Thm:pfdASP}, there is a $(\delta+3\epsilon)$-matching between $\B{M^\epsilon}$ and $\B{N^\epsilon}$.  Thus by \Cref{Lem:ModuleApproximations}\,(iii) and the triangle inequality for $\delta$-matchings (\cref{Rmk:MatchingsTriangleIneq}), there is a $(\delta+5\epsilon)$-matching between $\B{\rad M}$ and $\B{\rad N}$.  This induces a $(\delta+5\epsilon)$-matching between $\Bopen M=\Bopen {\rad M}$ and $\Bopen N=\Bopen {\rad N}$.

Since this is true for all $\epsilon>0$, we conclude that \[d_B(\Bopen M,\Bopen N)\leq d_I(M,N).\] Finally, the existence of a $\delta$-matching between $\Bopen M$ and $\Bopen N$ follows by \cref{Thm:matchingLimit} below.
\end{proof}

\begin{theorem}[{\cite[Theorem 4.10]{chazal2012structure}}]\label{Thm:matchingLimit}
If there is a $(\delta+\epsilon)$-matching between $\Bopen M$ and $\Bopen N$ for all $\epsilon>0$, then there is a $\delta$-matching between $\Bopen M$ and $\Bopen N$.
\end{theorem}
We note that the proof of \cref{Thm:matchingLimit} is self-contained; in invoking this result and its proof, we are not implicitly invoking any other results of \cite{chazal2012structure}.

\section{Converse Algebraic Stability}\label{Sec:ConverseToASP}
Since the algebraic stability theorem and its converse go hand in hand, for completeness we include a proof of the converse here.  We first present the result under the assumption that the persistence modules are interval-decomposbale, following \cite{lesnick2013theory}.  Together with \cref{Thm:pfdASP}, this establishes the isometry theorem for \pfd persistence modules, \cref{Thm:Isometry}. 
  
Following \cite{chazal2012structure}, we then establish a version of converse algebraic stability for q-tame persistence modules, as an easy corollary of the result in the interval-decomposable setting.  Together with \cref{Thm:ASP_ForQTame}, this gives the isometry theorem for q-tame persistence modules.

\subsection{Converse Algebraic Stability for Interval-Decomposable Persistence Modules}

\begin{lemma}\label{Lem:CyclicConverse} 
Let $\delta\geq 0$. 
\begin{itemize}
\item[(i)] If $\interval{ \birth}{\death }$, $\interval{ \birth'}{\death' }$ are intervals such that 
\begin{align*}
\interval{ \birth}{\death }&\subseteq \interval{ \birth'-\delta}{\death'+\delta }\textup,\\
\interval{ \birth'}{\death' }&\subseteq\interval{ \birth-\delta}{\death+\delta },
\end{align*}
then $C\interval{ \birth}{\death }$ and $C\interval{ \birth'}{\death' }$ are $\delta$-interleaved.
\item[(ii)] If $ \birth + \epsilon \geq \death $, then $C\interval{\birth}{\death }$ and the trivial module are $\delta$-interleaved.
\end{itemize}
\end{lemma}

\begin{proof} We leave the straightforward details to the reader. \end{proof} 

\begin{theorem}[{\cite[Theorem 3.3]{lesnick2013theory}}]\label{Thm:ConverseStabilityForDecomposableMods}
If $M$ and $N$ are interval-decomposable persistence modules and there exists a $\delta$-matching between $\B M$ and $\B N$, then $M$ and $N$ are $\delta$-interleaved.
In particular, \[d_B(\B M,\B N)\geq d_I(M,N).\]
\end{theorem}

\begin{proof}%
We may assume without loss of generality that 
\[M=\bigoplus_{I\in \B M} C(I),\qquad N=\bigoplus_{I\in \B N} C(I). \]
Let $\sigma:\B M\matching \B N$ be a $\delta$-matching, and let
\[
\begin{aligned}
M_\bullet=&\bigoplus_{I\in \dom(\sigma)} C(I),\\
M_\circ=&\bigoplus_{I \in \ker(\sigma)} C(I),
\end{aligned}
\quad
\begin{aligned}
N_\bullet=&\bigoplus_{I\in \im(\sigma)} C(I),\\
N_\circ=&\bigoplus_{I \in \coker(\sigma)} C(I),
\end{aligned}
\]
where $\ker(\sigma)$ and $\coker(\sigma)$ denote the complements of $\dom(\sigma)$ and $\im(\sigma)$ in $\B{M}$ and $\B{N}$, respectively. 

Clearly, $M=M_\bullet\oplus M_\circ$ and $N=N_\bullet\oplus N_\circ$.  By \cref{Lem:CyclicConverse}\,(i), for each $(I,J)\in \sigma$ we may choose a pair of $\delta$-interleaving morphisms 
\[f_I:C(I)\to C(J)(\delta),\qquad g_J:C(J)\to C(I)(\delta).\]
These morphisms induce a pair of $\delta$-interleaving morphisms
\[f_\bullet:M_\bullet\to N_\bullet(\delta),\qquad g_\bullet:N_\bullet\to M_\bullet(\delta).\]

Define a morphism $f:M\to N$ by taking the restriction of $f$ to $M_\bullet$ to be equal to $f_\bullet$ and taking the restriction of $f$ to $M_\circ$ to be the trivial morphism.  Symmetrically, define a morphism $g:N\to M$ by taking the restriction of $g$ to $N_\bullet$ to be equal to $g_\bullet$ and taking the restriction of $g$ to $N_\circ$ to be the trivial morphism.  By \cref{Lem:CyclicConverse}\,(ii), \[\translation M {2\delta} (M_\circ)=\translation N {2\delta} (N_\circ)=0.\]  From this fact and the fact that $f_\bullet$ and $g_\bullet$ are $\delta$-interleaving morphisms, it follows that $f$ and $g$ are $\delta$-interleaving morphisms as well.
\end{proof}

\vspace{0pt} %

\subsection{Converse Algebraic Stability for q-Tame Persistence Modules}\label{Sec:qTameConverse}
The converse algebraic stability theorem for q-tame persistence modules, as presented in \cite{chazal2012structure}, follows easily from \cref{Thm:ConverseStabilityForDecomposableMods}:

\begin{corollary}[{\cite[Theorem 4.11$''$]{chazal2012structure}}]\label{ConverseASTqTame}
For $M$ and $N$ q-tame persistence modules, \[d_B(\Bopen M,\Bopen N)\geq d_I(M,N).\]
\end{corollary}

\begin{proof}
Let $\delta=d_B(\Bopen M,\Bopen N)$.  By definition, $d_B(\Bopen {\rad M},\Bopen {\rad N})=\delta$.  Clearly, then, we also have that $d_B(\B {\rad M}, \B{\rad N})=\delta$, so by \cref{Thm:ConverseStabilityForDecomposableMods}, $d_I(\rad M,\rad N)\leq \delta$.  $d_I(M,\rad M)=0=d_I(N,\rad N),$ so by the triangle inequality, we conclude that $d_I(M,N)\leq \delta$.
\end{proof}

Together, \cref{Thm:ASP_ForQTame,ConverseASTqTame} give the isometry theorem for q-tame persistence modules:

\begin{theorem}[{\cite[Theorem 4.11]{chazal2012structure}}]
For $M$ and $N$ q-tame persistence modules, \[d_B(\Bopen M,\Bopen N)=d_I(M,N).\]
\end{theorem}

\section*{Acknowledgments}
We thank Vin de Silva, Peter Landweber, Daniel M\"{u}llner, and the anonymous referees for valuable feedback on this work; this paper has benefitted from their input in several ways.  We also thank Guillaume Troianowski and Daniel for providing us with a copy of their manuscript \cite{mullner2011proximity} and for helpful correspondence regarding their related work.  In addition, we
thank William Crawley-Boevey and Vin for enlightening discussions about q-tame persistence modules.  UB was partially supported by the {\sc Toposys} project FP7-ICT-318493-STREP.  ML thanks the Institute for Advanced Study and the Institute for Mathematics and its Applications for their support and hospitality during the course of this project.  ML was partially supported by NSF grant DMS-1128155.  Any opinions, findings, and conclusions or recommendations expressed in this material are those of the authors and do not necessarily reflect the views of the National Science Foundation.

\bibliography{Stability_Paper_Refs}

\end{document}